\documentclass{amsart}

\usepackage{hyperref}

\newcommand*{\mailto}[1]{\href{mailto:#1}{\nolinkurl{#1}}}
\newcommand{\arxiv}[1]{\href{http://arxiv.org/abs/#1}{arXiv:#1}}

\newcommand{\msc}[1]{\href{http://www.ams.org/msc/msc2010.html?t=&s=#1}{#1}}

\newtheorem{theorem}{Theorem}[section]

\newtheorem{lemma}[theorem]{Lemma}
\newtheorem{proposition}[theorem]{Proposition}
\newtheorem{corollary}[theorem]{Corollary}

\newtheorem{remark}[theorem]{Remark}
\newtheorem{example}{Example}[section]

\newcommand{\R}{{\mathbb R}}
\newcommand{\N}{{\mathbb N}}

\newcommand{\C}{{\mathbb C}}

\newcommand{\spr}[2]{\langle #1 , #2 \rangle}

\newcommand{\Qr}{\mathsf{q}}

\newcommand{\Wr}{\mathsf{w}}

\newcommand{\T}{\mathrm{T}}

\newcommand{\F}{\mathcal{F}}
\newcommand{\Sr}{\mathrm{s}}

\newcommand{\E}{\mathrm{e}}
\newcommand{\I}{\mathrm{i}}

\newcommand{\supp}{\mathrm{supp}}

\newcommand{\im}{\mathrm{Im}}
\newcommand{\re}{\mathrm{Re}}
\newcommand{\loc}{{\mathrm{loc}}}
\newcommand{\cc}{{\mathrm{c}}}

\newcommand{\sing}{{\mathrm{s}}}

\newcommand{\OO}{\mathcal{O}}

\newcommand{\ledot}{\,\cdot\,}
\newcommand{\redot}{\cdot\,}

\newcommand{\NL}{(0-)}
\newcommand{\NLz}{(z,0-)}
\newcommand{\qd}{{[1]}}

\newcommand{\Hast}{\dot{H}^1[0,L)}
\newcommand{\Hasto}{\dot{H}^1_0[0,L)}
\newcommand{\cH}{\mathcal{H}}
\newcommand{\Lmu}{L^2(\R;\mu)}
\newcommand{\dip}{\upsilon}
\newcommand{\eps}{\varepsilon}

\renewcommand{\labelenumi}{(\roman{enumi})}
\renewcommand{\theenumi}{\roman{enumi}}

\numberwithin{equation}{section}

\begin{document}

\title[Absolutely continuous spectrum]{On the absolutely continuous spectrum of~generalized~indefinite strings II}

\author[J.\ Eckhardt]{Jonathan Eckhardt}
\address{Department of Mathematical Sciences\\ Loughborough University\\ Loughborough LE11 3TU \\ UK\\  and  Faculty of Mathematics\\ University of Vienna\\ Oskar-Morgenstern-Platz 1\\ 1090 Wien\\ Austria}
\email{\mailto{J.Eckhardt@lboro.ac.uk};\ \mailto{jonathan.eckhardt@univie.ac.at}}
\urladdr{\url{http://homepage.univie.ac.at/jonathan.eckhardt/}}

\author[A.\ Kostenko]{Aleksey Kostenko}
\address{Faculty of Mathematics and Physics\\ University of Ljubljana\\ Jadranska 21\\ 1000 Ljubljana\\ Slovenia\\ and Faculty of Mathematics\\ University of Vienna\\ Oskar-Morgenstern-Platz 1\\ 1090 Wien\\ Austria}
\email{\mailto{Aleksey.Kostenko@fmf.uni-lj.si};\ \mailto{Oleksiy.Kostenko@univie.ac.at}}
\urladdr{\url{http://www.mat.univie.ac.at/~kostenko/}}
 
\author[T.\ Kukuljan]{Teo Kukuljan}
\address{Faculty of Mathematics and Physics\\ University of Ljubljana\\ Jadranska 21\\ 1000 Ljubljana\\ Slovenia}
\email{\mailto{teo.kukuljan@student.fmf.uni-lj.si}}

%\thanks{... (to appear)}
%\thanks{\href{http://dx.doi.org/...}{... (to appear)}}
%\thanks{\href{http://dx.doi.org/...}{... {\bf X} (20XX), no.~X, pp--pp}}

\thanks{{\it Research supported by the Austrian Science Fund (FWF) under Grants No.~P29299 (J.E.) and P28807 (A.K.) as well as by the Slovenian Research Agency (ARRS) under Grant No.\ J1-9104 (A.K.)}}

\keywords{Absolutely continuous spectrum, generalized indefinite strings}
\subjclass[2010]{Primary \msc{34L05}, \msc{34B07}; Secondary \msc{34L25}, \msc{37K15}}

\begin{abstract}
We continue to investigate absolutely continuous spectrum of generalized indefinite strings. 
By following an approach of Deift and Killip, we establish stability of the absolutely continuous spectra of two more model examples of generalized indefinite strings under rather wide perturbations. 
In particular, one of these results allows us to prove that the absolutely continuous spectrum of the isospectral problem associated with the two-component Camassa--Holm system in a certain dispersive regime is essentially supported on the set $(-\infty,-1/2]\cup [1/2,\infty)$.
\end{abstract}

\maketitle

\section{Introduction}

In this paper, we continue our study of the absolutely continuous part of the spectrum of {\em generalized indefinite strings} initiated in \cite{ACSpec}. 
We recall briefly (see Section~\ref{sec:prelim} for further details) that a generalized indefinite string is a triple $(L,\omega,\dip)$ such that $L\in(0,\infty]$, $\omega$ is a real distribution in $H^{-1}_{\loc}[0,L)$ and $\dip$ is a non-negative Borel measure on the interval $[0,L)$. 
 Associated with such a triple is the ordinary differential equation  
  \begin{align}\label{eqnDEIndStr}
  -f'' = z\, \omega f + z^2 \dip f
 \end{align}
 on $[0,L)$, where $z$ is a complex spectral parameter. 
 Spectral problems of this type are of interest for at least two reasons: 
Firstly, they constitute a canonical model for operators with simple spectrum; see \cite{IndefiniteString, IndMoment}. 
Secondly, they are of relevance in connection with certain completely integrable nonlinear wave equations, for which these kinds of spectral problems arise as isospectral problems.
Arguably the most prominent examples hereby are the Camassa--Holm equation \cite{caho93} and its two-component generalization \cite{coiv08}
\begin{align}
 u_t - u_{xxt} & = 2u_x u_{xx}  - 3uu_x  +  u u_{xxx} - \rho \rho_x, &
  \rho_t & = - (u\rho)_x.
\end{align} 
 Our main result in this regard can be deduced readily from Theorem~\ref{thmApp2CH} by standard arguments:

\begin{theorem}\label{thm2CHac}
Let $u$ be a real-valued function in $H^1(\R)$ and let $\dip$ be a non-negative Borel measure on $\R$ such that its singular part is finite and $\rho-1$ belongs to $L^2(\R)$, where $\rho$ is the square root of the Radon--Nikod\'ym derivative of $\dip$.
Then the essential spectrum of the spectral problem
\begin{align}\label{eq:CHspec}
 - g'' + \frac{1}{4}g & = z\,\omega\, g + z^2 \dip\, g,  & \omega & = u - u'',
\end{align}
 coincides with the set $(-\infty,-1/2]\cup[1/2,\infty)$ and the absolutely continuous spectrum is of multiplicity two and essentially supported on $(-\infty,-1/2]\cup[1/2,\infty)$. 
\end{theorem}

\noindent
Spectral problems of the form \eqref{eq:CHspec} are of the same importance for the conservative Camassa--Holm flow \cite{chlizh06, coiv08, ConservCH, LagrangeCH, ConservMP, hoiv11} as the one-dimensional Schr\"odinger operator for the Korteweg--de Vries equation. 
Theorem~\ref{thm2CHac} is relevant because it covers a natural phase space for the two-component  Camassa--Holm system; see also \cite{grhora12}. 
 
 Results on spectral types, even for the special case of a Krein string
   \begin{align}\label{eqnClaS}
  -f'' = z\, \omega f,
 \end{align}
that is, when $\omega$ is a non-negative Borel measure and $\dip$ vanishes identically, are rather scarce (however, let us mention the two recent articles \cite{bd17} and \cite{bd19} by Bessonov and Denisov). 
 The main reason for this lies in the fact that the spectral parameter appears in the {\em wrong} place, which does not allow to view~\eqref{eqnClaS} as an additive perturbation immediately.  
Our approach to the absolutely continuous spectrum of generalized indefinite strings follows \cite{ACSpec} and is based on the elegant ideas of Deift and Killip \cite{deki99}. 
In order to implement this approach, we need two main ingredients:
The first ingredient is a continuity property for the correspondence between generalized indefinite strings and their associated Weyl--Titchmarsh functions  (see Section~\ref{sec:prelim} for more details) obtained in \cite[Proposition~6.2]{IndefiniteString}. 
The second ingredient is a so-called {\em dispersion relation} or {\em trace formula}, which provides a relation between the spectral/scattering data and the coefficients in the differential equation, and hence allows to {\em control} the spectral measure by means of the coefficients. 
Such trace formulas are also of interest in connection with the aforementioned nonlinear wave equations because they give rise to conserved quantities of the flow. 
However, even for the two-component Camassa--Holm system, such trace formulas have not yet been stated explicitly in the literature to the best of our knowledge (although they can be derived from \cite{hoiv11,iv06} for sufficiently smooth coefficients). 
Due to the low regularity of our coefficients, the derivation of corresponding trace formulas will require more efforts, even compared to the cases considered in \cite{ACSpec}.

In conclusion, let us sketch the content of the article. 
Section~\ref{sec:prelim} is of preliminary character and collects necessary notions and facts from the spectral theory of generalized indefinite strings. 
Section~\ref{secMR} contains the statements of our main results (Theorem~\ref{thm0} and Theorem~\ref{thmalpha}) about the absolutely continuous spectrum for certain classes of generalized indefinite strings, which are rather strong perturbations of explicitly solvable models (Example~\ref{exa0} and Example~\ref{exaalpha}).
Although we will not present the necessary details here, these perturbations can indeed be interpreted in a certain way as additive perturbations, which are only of Hilbert--Schmidt class in general however. 
Our main theorems will be proved in Sections~\ref{secPr1} and~\ref{secPr2} respectively. 
 Even though we will generally follow the approach of \cite{ACSpec} based on the method by Deift and Killip from \cite{deki99}, the necessary ingredients are not readily available for the present class of coefficients and need to be established first. 
In the final section, we will then show how corresponding results for the spectral problem~\eqref{eq:CHspec} can be derived readily from Theorem~\ref{thmalpha}.

%%%%%%%%%%%%%%%%%%%%%%%%%%%%%%
\section{Generalized indefinite strings}\label{sec:prelim} 
%%%%%%%%%%%%%%%%%%%%%%%%%%%%%%

We will first introduce several spaces of functions and distributions.  
 For every fixed $L\in(0,\infty]$, we denote with $H^1_{\loc}[0,L)$, $H^1[0,L)$ and $H^1_{\cc}[0,L)$ the usual Sobolev spaces. 
 To be precise, this means  
\begin{align}
H^1_{\loc}[0,L) & =  \lbrace f\in AC_{\loc}[0,L) \,|\, f'\in L^2_{\loc}[0,L) \rbrace, \\
 H^1[0,L) & = \lbrace f\in H^1_{\loc}[0,L) \,|\, f,\, f'\in L^2[0,L) \rbrace, \\ 
 H^1_{\cc}[0,L) & = \lbrace f\in H^1[0,L) \,|\, \supp(f) \text{ compact in } [0,L) \rbrace.
\end{align}
The space of distributions $H^{-1}_{\loc}[0,L)$ is the topological dual of $H^1_{\cc}[0,L)$. 
One notes that the mapping $\Qr\mapsto\chi$, defined by
 \begin{align}
    \chi(h) = - \int_0^L \Qr(x)h'(x)dx, \quad h\in H^1_{\cc}[0,L),
 \end{align} 
 establishes a one-to-one correspondence between $L^2_{\loc}[0,L)$ and $H^{-1}_{\loc}[0,L)$. 
The unique function $\Qr\in L^2_{\loc}[0,L)$ corresponding to some distribution $\chi\in H^{-1}_{\loc}[0,L)$ in this way will be referred to as {\em the normalized anti-derivative} of $\chi$.
 We say that a distribution in $H^{-1}_{\loc}[0,L)$ is {\em real} if its normalized anti-derivative is real-valued almost everywhere on $[0,L)$.  

A particular kind of distribution in $H^{-1}_{\loc}[0,L)$ arises from Borel measures on the interval $[0,L)$.
 In fact, if $\chi$ is a complex-valued Borel measure on $[0,L)$, then we will identify it with the distribution in $H^{-1}_{\loc}[0,L)$ given by  
 \begin{align}
  h \mapsto \int_{[0,L)} h\,d\chi. 
 \end{align}
The normalized anti-derivative $\Qr$ of such a $\chi$ is simply given by the left-continuous distribution function 
 \begin{align}
 \Qr(x)=\int_{[0,x)}d\chi
 \end{align}
 for almost all $x\in [0,L)$, as an integration by parts (use, for example, \cite[Exercise~5.8.112]{bo07}, \cite[Theorem~21.67]{hest65}) shows.   

 In order to obtain a self-adjoint realization of the differential equation~\eqref{eqnDEIndStr} in a suitable Hilbert space later, we also introduce the function space  
\begin{align}
\Hast & = \begin{cases} \lbrace f\in H^1_{\loc}[0,L) \,|\, f'\in L^2[0,L),~ \lim_{x\rightarrow L} f(x) = 0 \rbrace, & L<\infty, \\ \lbrace f\in H^1_{\loc}[0,L) \,|\, f'\in L^2[0,L) \rbrace, & L=\infty, \end{cases} 
\end{align}
 as well as the linear subspace  
\begin{align}
 \Hasto & = \lbrace f\in \Hast \,|\, f(0) = 0 \rbrace, 
\end{align}
 which turns into a Hilbert space when endowed with the scalar product 
 \begin{align}\label{eq:normti}
 \spr{f}{g}_{\Hasto} = \int_0^L f'(x) g'(x)^\ast dx, \quad f,\, g\in\Hasto.
 \end{align}
Here and henceforth, we will use a star to denote complex conjugation.  
The space $\Hasto$ can be viewed as a completion with respect to the norm induced by~\eqref{eq:normti} of the space of all smooth functions which have compact support in $(0,L)$. 
In particular, the space $\Hasto$ coincides algebraically and topologically with the usual Sobolev space $H^1_0[0,L)$ when $L$ is finite. 

A generalized indefinite string is a triple $(L,\omega,\dip)$ such that $L\in(0,\infty]$, $\omega$ is a real distribution in $H^{-1}_{\loc}[0,L)$ and $\dip$ is a non-negative Borel measure on the interval $[0,L)$.  
The corresponding normalized anti-derivative of the distribution $\omega$ will always be denoted with $\Wr$ in the following. 
 Associated with such a generalized indefinite string is the inhomogeneous differential equation
 \begin{align}\label{eqnDEinho}
  -f''  = z\, \omega f + z^2 \dip f + \chi, 
 \end{align}
 where $\chi$ is a distribution in $H^{-1}_{\loc}[0,L)$ and $z$ is a complex spectral parameter. 
 Of course, this differential equation has to be understood in a weak sense:   
  A solution of~\eqref{eqnDEinho} is a function $f\in H^1_{\loc}[0,L)$ such that 
 \begin{align}
  \Delta_f h(0) + \int_{0}^L f'(x) h'(x) dx = z\, \omega(fh) + z^2 \dip(fh) +  \chi(h), \quad h\in H^1_{\cc}[0,L),
 \end{align}
 for some constant $\Delta_f\in\C$. 
 In this case, the constant $\Delta_f$ is uniquely determined and will be denoted with $f'\NL$ for apparent reasons. 
 
 With this notion of solution, we are able to introduce the fundamental system of solutions $\theta(z,\redot)$, $\phi(z,\redot)$ of the homogeneous differential equation
   \begin{align}\label{eqnDEho}
  -f'' = z\, \omega f + z^2 \dip f
 \end{align}
  satisfying the initial conditions
 \begin{align}
  \theta(z,0)& = \phi'\NLz =1, &  \theta'\NLz & = \phi(z,0) =0,
 \end{align}
 for every $z\in\C$; see \cite[Lemma~3.2]{IndefiniteString}. 
 Even though the derivatives of these functions are only locally square integrable in general, there are unique left-continuous functions $\theta^\qd(z,\redot)$, $\phi^\qd(z,\redot)$ on $[0,L)$ such that 
 \begin{align}
   \theta^\qd(z,x) & = \theta'(z,x) + z\Wr(x)\theta(z,x), & \phi^\qd(z,x) & = \phi'(z,x) + z\Wr(x)\phi(z,x),
 \end{align}
 for almost all $x\in[0,L)$;  see \cite[Equation~(4.12)]{IndefiniteString}.
 These functions will henceforth be referred to as {\em quasi-derivatives} of the solutions $\theta(z,\redot)$, $\phi(z,\redot)$. 
  It follows from \cite[Corollary~3.5]{IndefiniteString} that the functions 
   \begin{align}
   z & \mapsto \theta(z,x), & z & \mapsto \theta^\qd(z,x), & z & \mapsto \phi(z,x), & z & \mapsto \phi^\qd(z,x), 
 \end{align}
 are real entire for every fixed $x\in[0,L)$.
  At the origin, when $z$ is zero, one readily infers that our fundamental system is given explicitly by
 \begin{align}\label{eqncsatzero}
   \theta(0,x) & = 1, & \theta^\qd(0,x) & = 0, & \phi(0,x) & = x, & \phi^\qd(0,x) & = 1, 
 \end{align}
 for all $x\in[0,L)$. 
 More crucially, we will furthermore require the following formulas for the derivatives of these functions with respect to the spectral parameter at the origin, proved in \cite[Section~4]{ACSpec}. 
 Here, let us note that differentiation with respect to the spectral parameter will be denoted with a dot and is always meant to be done after taking quasi-derivatives.
 
 \begin{proposition}\label{propFS}
  For every $x\in[0,L)$, we have   
  \begin{align}
    \label{eqndtheta0} \dot{\theta}(0,x) & =  - \int_0^x \Wr(t)dt, &   \dot{\theta}^\qd(0,x) & =  0, \\
    \label{eqndphi0} \dot{\phi}(0,x) & = \int_0^x \int_0^t \Wr(s)ds\, dt - \int_0^x \Wr(t) t\, dt, &   \dot{\phi}^\qd(0,x) & = \int_0^x \Wr(t)dt,
  \end{align}
  as well as  
  \begin{align}
    \label{eqnddtheta0}  \ddot{\theta}(0,x) & = \biggl(\int_0^x \Wr(t)dt\biggr)^2 - 2 \int_0^x \int_0^t \Wr(s)^2 ds\,dt - 2 \int_0^x \int_{[0,t)} d\dip\,dt,                                                   \\
    \label{eqnddthetap0} \ddot{\theta}^\qd(0,x) & = -2 \int_0^x \Wr(t)^2 dt -2 \int_{[0,x)} d\dip,                                                    \\
    \label{eqnddphip0} \ddot{\phi}^\qd(0,x) & = \biggl(\int_0^x \Wr(t)dt\biggr)^2 - 2 \int_0^x \Wr(t)^2 t\, dt -2 \int_{[0,x)} t\, d\dip(t).                                                 
  \end{align}
 \end{proposition}

 A generalized indefinite string $(L,\omega,\dip)$ gives rise to a self-adjoint linear relation in a suitable Hilbert space. 
 In order to introduce this object, we consider the space 
 \begin{align}
 \cH = \Hasto\times L^2([0,L);\dip),
\end{align}
which turns into a Hilbert space when endowed with the scalar product
\begin{align}
 \spr{f}{g}_{\cH} = \int_0^L f_1'(x) g_1'(x)^\ast dx + \int_{[0,L)} f_2(x) g_2(x)^\ast d\dip(x), \quad f,\, g\in \cH.
\end{align} 
 The respective components of some vector $f\in\cH$ are hereby always denoted by adding subscripts, that is, with $f_1$ and $f_2$.  
Now the linear relation $\T$ in the Hilbert space $\cH$ is defined by saying that some pair $(f,g)\in\cH\times\cH$ belongs to $\T$ if and only if  the two equations 
\begin{align}\label{eqnDEre1}
-f_1'' & =\omega g_{1} + \dip g_{2}, &  \dip f_2 & =\dip g_{1},
\end{align}
hold. 
In order to be precise, the right-hand side of the first equation in~\eqref{eqnDEre1} has to be understood as the $H^{-1}_{\loc}[0,L)$ distribution given by 
\begin{align}
 h \mapsto \omega(g_1h) + \int_{[0,L)} g_2 h\, d\dip. 
\end{align} 
 Moreover, the second equation in~\eqref{eqnDEre1} holds if and only if $f_2$ is equal to $g_1$ almost everywhere on $[0,L)$ with respect to the measure $\dip$. 
  The linear relation $\T$ turns out to be self-adjoint in the Hilbert space $\cH$; see \cite[Theorem~4.1]{IndefiniteString}.
 
   A central object in the spectral theory for the linear relation $\T$ is the associated {\em Weyl--Titchmarsh function} $m$. 
   This function can be defined on $\C\backslash\R$ by 
 \begin{align}\label{eqnmdef}
  m(z) =  \frac{\psi'\NLz}{z\psi(z,0)},\quad z\in\C\backslash\R,
 \end{align} 
 where $\psi(z,\redot)$ is the unique (up to constant multiples) non-trivial solution of the differential equation~\eqref{eqnDEho} which lies in $\Hast$ and $L^2([0,L);\dip)$, guaranteed to exist by \cite[Lemma~4.2]{IndefiniteString}. 
 It has been shown in \cite[Lemma~5.1]{IndefiniteString} that the Weyl--Titchmarsh function $m$ is a Herglotz--Nevanlinna function, that is, it is analytic, maps the upper complex half-plane $\C_+$ into the closure of the upper complex half-plane and satisfies the symmetry relation
 \begin{align}\label{eqnHNsym}
  m(z)^\ast = m(z^\ast), \quad z\in\C\backslash\R. 
 \end{align}
For this reason, the Weyl--Titchmarsh function $m$ admits an integral representation, which takes the form 
\begin{align}\label{eqnWTmIntRep}
 m(z) = c_1 z + c_2 - \frac{1}{Lz} +  \int_\R \frac{1}{\lambda-z} - \frac{\lambda}{1+\lambda^2}\, d\mu(\lambda), \quad z\in\C\backslash\R, 
\end{align}
for some constants $c_1$, $c_2\in\R$ with $c_1\geq0$ and a non-negative Borel measure $\mu$ on $\R$ with $\mu(\lbrace0\rbrace)=0$ for which the integral  
\begin{align}
 \int_\R \frac{d\mu(\lambda)}{1+\lambda^2} 
\end{align}
is finite. 
Here we employ the convention that whenever an $L$ appears in a denominator, the corresponding fraction has to be interpreted as zero if $L$ is infinite.

The measure $\mu$ turns out to be a {\em spectral measure} for the linear relation $\T$ in the sense that the operator part of $\T$ is unitarily equivalent to multiplication with the independent variable in $\Lmu$; see \cite[Theorem~5.8]{IndefiniteString}.
 Of course, this establishes an immediate connection between the spectral properties of the linear relation $\T$ and the measure $\mu$. 
 For example, the spectrum of $\T$ coincides with the topological support of $\mu$ and thus can be read off the singularities of $m$ (more precisely, the function $m$ admits an analytic continuation away from the spectrum of $\T$). 

For the sake of simplicity, we shall always mean the spectrum of the corresponding linear relation when we speak of the spectrum of a generalized indefinite string in the following.
The same convention applies to the various spectral types.

%%%%%%%%%%%%%%%%%%%%%%%
 \section{Absolutely continuous spectrum}\label{secMR}
%%%%%%%%%%%%%%%%%%%%%%%

In general, any kind of (simple) spectrum can arise from a generalized indefinite string; see \cite[Theorem~6.1]{IndefiniteString}.
Here we are interested in the absolutely continuous spectrum of a particular class of generalized indefinite strings, which are suitable perturbations of the following explicitly solvable case. 
  
  \begin{example}\label{exa0}
   Let $S_0$ be the generalized indefinite string $(L_0,\omega_0,\dip_0)$ such that $L_0$ is infinite, the distribution $\omega_0$ is identically zero and the measure $\dip_0$ is the Lebesgue measure on $[0,\infty)$.
   We note that under these assumptions, the corresponding differential equation~\eqref{eqnDEho} simply reduces to 
       \begin{align}
       - f'' = z^2f.
       \end{align}
  For every $z\in\C\backslash\R$, the function $\psi_0(z,\redot)$  given by\footnote{We denote with $\C_+$ the (open) upper complex half-plane and with $\C_-$ the (open) lower complex half-plane.}
    \begin{align}\label{eq:psi_0}
      \psi_0(z,x) = \begin{cases} \E^{\I zx}, &  x\in[0,\infty),~z\in\C_+,  \\ \E^{-\I zx}, & x\in[0,\infty),~ z\in\C_-, \end{cases}
    \end{align}   
    is a solution of this differential equation which lies in $\dot{H}^1[0,\infty)$ and $L^2[0,\infty)$.
    Consequently,  the corresponding Weyl--Titchmarsh function $m_0$ is given explicitly by 
    \begin{align}
      m_0(z) = \frac{\psi_0'\NLz}{z\psi_0(z,0)} = \begin{cases} \I, & z\in\C_+, \\ -\I, & z\in\C_-. \end{cases}
    \end{align}
    This guarantees that the spectrum of $S_0$ is purely absolutely continuous and coincides with $\R$. 
  \end{example}
  
   In particular, the essential spectrum of $S_0$ coincides with the real line $\R$ and the absolutely continuous spectrum of $S_0$ is essentially supported on $\R$.
 The latter means that every subset of $\R$ with positive Lebesgue measure has positive measure with respect to the corresponding spectral measure $\mu_0$.
  It turns out that these two properties continue to hold under a rather wide class of perturbations. 
  
    \begin{theorem}\label{thm0}
    Let $S$ be a generalized indefinite string $(L,\omega,\dip)$ such that $L$ is infinite and
     \begin{align}\label{eqnCondS0}
\int_0^\infty \left|\Wr(x) - c\right|^2 dx +  \int_0^\infty \left|\rho(x) - \eta\right|^2 dx + \int_{[0,\infty)} d\dip_\sing <\infty
    \end{align}
    for a real constant $c$ and a positive constant $\eta$, where $\Wr$ is the normalized anti-derivative of $\omega$, $\rho$ is the square root of the Radon--Nikod\'ym derivative of $\dip$ and $\dip_\sing$ is the singular part of $\dip$  (both with respect to the Lebesgue measure). 
    Then the essential spectrum of $S$ coincides with the real line $\R$ and the absolutely continuous spectrum of $S$ is essentially supported on $\R$. 
    \end{theorem}
    
A proof for this result will be given in Section~\ref{secPr1}. 
In view of the applications we have in mind (see Section~\ref{secAPP}), we are furthermore interested in perturbations of another explicitly solvable case involving a positive parameter $\alpha$.
  
  \begin{example}\label{exaalpha}
    Let $S_\alpha$ be the generalized indefinite string $(L_\alpha,\omega_\alpha,\dip_\alpha)$ such that $L_\alpha$ is infinite, the distribution $\omega_\alpha$ is identically zero and the measure $\dip_\alpha$ is given by 
    \begin{align}
       \int_{B}d\dip_\alpha = \int_B \frac{1}{(1+2 \alpha x)^2} dx
    \end{align} 
    for every Borel set $B\subseteq[0,\infty)$. 
    We note that under these assumptions, the corresponding differential equation~\eqref{eqnDEho} simply reduces to 
    \begin{align}\label{eq:SP2}
    - f''(x) = \frac{z^2}{(1+2 \alpha x)^{2}}  f(x),\quad x\in[0,\infty).
    \end{align}
  For every $z\in\C_+\cup(-\alpha,\alpha)\cup\C_-$, the function $\psi_\alpha(z,\redot)$ given by\footnote{In the following, we will always take the branch of the square root $\sqrt{\cdot}$ with cut along the positive semi-axis $[0,\infty)$ defined by $\sqrt{z}=\sqrt{|z|}\E^{\I \arg(z)/2}$ with $\arg(z)\in [0,2\pi)$.}  
    \begin{align}\label{eq:psiC}
      \psi_\alpha(z,x) = (1+2\alpha x)^{\frac{\I \sqrt{z^2-\alpha^2}}{2{\alpha}} + \frac{1}{2}}, \quad x\in[0,\infty),
    \end{align}   
    is a solution of this differential equation which lies in $\dot{H}^1[0,\infty)$ and $L^2([0,\infty);\dip_\alpha)$.
    Let us point out here that the mapping 
    \begin{align}
      z\mapsto\sqrt{z^2-\alpha^2}
    \end{align}
    is analytic and has positive imaginary part on the domain $\C_+\cup(-\alpha,\alpha)\cup\C_-$.
    Moreover, near the real line one has
 \begin{align}\label{eq:konR}
   \lim_{\varepsilon\rightarrow0} \sqrt{(\lambda\pm\varepsilon\I)^2 - \alpha^2} = \begin{cases}
  \mp\sqrt{\lambda^2-\alpha^2},  & \lambda \in (-\infty,-\alpha], \\
  \ \I \sqrt{\alpha^2-\lambda^2}, & \lambda \in (-\alpha, \alpha), \\
  \pm\sqrt{\lambda^2-\alpha^2}, & \lambda \in[\alpha,\infty). \end{cases}
  \end{align}
    Now the corresponding Weyl--Titchmarsh function $m_\alpha$ is given explicitly by  
    \begin{align}\label{eq:WTfB}
      m_\alpha(z) = \frac{\psi_\alpha'\NLz}{z\psi_\alpha(z,0)} = \frac{\I \sqrt{z^2-\alpha^2}+{\alpha}}{z}, \quad z\in\C\backslash\R.  % = \frac{z}{-\I \sqrt{z^2-\alpha^2}+{\alpha}}
    \end{align}
    This guarantees that the spectrum of $S_\alpha$ is purely absolutely continuous and coincides with the set  $ (-\infty,-\alpha] \cup [\alpha,\infty)$. 
  \end{example}
  
    In particular, the essential spectrum of $S_\alpha$ coincides with  the set  $ (-\infty,-\alpha] \cup [\alpha,\infty)$ and the absolutely continuous spectrum of $S_\alpha$ is essentially supported on $ (-\infty,-\alpha] \cup [\alpha,\infty)$.
  These two properties are again preserved under a rather wide class of perturbations. 
  
  \begin{theorem}\label{thmalpha}
    Let $S$ be a generalized indefinite string $(L,\omega,\dip)$ such that $L$ is infinite and 
  \begin{align}\label{eqnCondSalpha}
      \int_0^\infty   \left| \Wr(x) - c \right|^2 x\, dx  +  \int_0^\infty   \Bigl| \rho(x) - \frac{\eta}{1+2{\alpha}x} \Bigr|^2 x\, dx + \int_{[0,\infty)}  x\, d\dip_\sing(x) & < \infty
    \end{align}
    for a real constant $c$ and positive constants $\alpha$ and $\eta$, where $\Wr$ is the normalized anti-derivative of $\omega$, $\rho$ is the square root of the Radon--Nikod\'ym derivative of $\dip$ and $\dip_\sing$ is the singular part of $\dip$  (both with respect to the Lebesgue measure). 
    Then the essential spectrum of $S$ coincides with the set $(-\infty,-{\alpha}/{\eta}] \cup [{\alpha}/{\eta},\infty)$ and the absolutely continuous spectrum of $S$ is essentially supported on $(-\infty,-{\alpha}/{\eta}]\cup [{\alpha}/{\eta},\infty)$. 
  \end{theorem}
  
   Although the proof of this result is quite similar to the one for Theorem~\ref{thm0} in principle, it will be carried out separately in Section~\ref{secPr2} due to differences in details.   
   
\begin{remark}
 In the special case when the distribution $\omega$ is identically zero, the spectral problem~\eqref{eqnDEho} reduces to a Krein string with a squared spectral parameter. 
 By exploiting this observation, Theorem~\ref{thm0} and Theorem~\ref{thmalpha} also give new results about the absolutely continuous spectrum for a class of Krein strings that have not been covered previously in~\cite{ACSpec}. 
\end{remark}

%%%%%%%%%%%%%%%%%%%%%%%
 \section{Proof of Theorem~\ref{thm0}}\label{secPr1}
%%%%%%%%%%%%%%%%%%%%%%%

  To begin with, let us consider a particular class of generalized indefinite strings. 
  We assume that $(L,\omega,\dip)$ is a generalized indefinite string such that $L$ is infinite and there is an $R>0$ such that the normalized anti-derivative $\Wr$ of $\omega$ is equal to zero almost everywhere on $[R,\infty)$ and the measure $\dip$ coincides with the Lebesgue measure on $[R,\infty)$.
  Under these assumptions, for every $z$ in the upper complex half-plane $\C_+$, there is a {\em Jost solution} $f(z,\redot)$ of the differential equation~\eqref{eqnDEho} such that 
  \begin{align}
    f(z,x) =  \E^{\I z x}, \quad x\in[R,\infty). 
  \end{align}   
  We note that since the function $f(z,\redot)$ clearly lies in $\dot{H}^1[0,\infty)$ and $L^2([0,\infty);\dip)$, the corresponding Weyl--Titchmarsh function $m$ is given by 
  \begin{align}\label{eqnmJost}
    m(z) = \frac{f'(z,0-)}{z f(z,0)}, \quad z\in\C_+.
  \end{align}
  
Next, we define the function $a$ on $\C_+$ via 
   \begin{align}\label{eq:a_0}
     a(z) = \frac{\I z f(z,0)+f'(z,0-)}{2\I z}, \quad z\in\C_+,
   \end{align}
   which can be viewed as the reciprocal transmission coefficient when the differential equation is suitably extended to the full line.
 It has a unique continuation (denoted with $a$ as well for simplicity) to an entire function. 
In fact, if $\theta$, $\phi$ denotes the fundamental system of solutions of the differential equation~\eqref{eqnDEho} as in Section~\ref{sec:prelim}, then we may write
\begin{align}
  f(z,x) = f(z,0)\theta(z,x) + f'(z,0-)\phi(z,x), \quad x\in[0,\infty),~z\in\C_+.
\end{align}
Upon evaluating this function and its derivative at a point $r> R$, we get 
\begin{align}
 \E^{\I zr} & = f(z,0)\theta(z,r) + f'(z,0-)\phi(z,r), \\
\I z\E^{\I zr} & = f(z,0)\theta^\qd(z,r) + f'(z,0-)\phi^\qd(z,r),
\end{align}  
which we can solve for $f(z,0)$ and $f'(z,0-)$ to obtain
\begin{align}
f(z,0) & = \E^{\I z r}(\phi^\qd(z,r) - \I z\phi(z,r)), \\ f'(z,0-) & =  \E^{\I zr}(\I z\theta(z,r) - \theta^\qd(z,r)).
\end{align}
Plugging this into the definition of $a$ shows that
\begin{align}\label{eqnaasthephi}
    2 \E^{-\I z r} a(z)  = \theta(z,r) - \frac{1}{\I z} \theta^\qd(z,r)   - \I z \phi(z,r) + \phi^\qd(z,r), \quad z\in\C_+. 
\end{align}
    Upon taking into account the second equality in~\eqref{eqncsatzero}, this identity guarantees that the function $a$ has a unique continuation to an entire function. 
    
     In a similar way, we introduce the function $b$ on $\C_+$ via 
  \begin{align}\label{eq:b_0}
    b(z) = \frac{\I z f(z,0) - f'(z,0-)}{2\I z}, \quad z\in\C_+.
  \end{align} 
  As before, we see that $b$ can be continued to an entire function because of the identity 
   \begin{align}\label{eqnbasthephi}
     2\E^{-\I z r} b(z)  = - \theta(z,r) + \frac{1}{\I z} \theta^\qd(z,r)
  - \I z \phi(z,r) + \phi^\qd(z,r), \quad z\in\C_+,
   \end{align}
   which holds as long as $r>R$. 
  Now it is a straightforward computation to verify that for real $\lambda$, we have  
  \begin{align}\label{eqnabone}
    |a(\lambda)|^2  = |b(\lambda)|^2 + 1,
  \end{align}
  which guarantees that $a$ has no zeros on the real axis. 
  Furthermore, there are no zeros in the upper complex half-plane.   
  In fact, if $z$ was a zero in $\C_+$, then 
  \begin{align}
     m(z) = \frac{f'(z,0-)}{z f(z,0)} = \frac{-\I z f(z,0)}{z f(z,0)} = -\I,
  \end{align} 
  which is a contradiction to the fact that $m$ is a Herglotz--Nevanlinna function.   
 
\begin{lemma}\label{lem:asymptot}
 For every $r>R$, there is an $\varepsilon>0$ such that 
\begin{align}\label{eq:asymptot}
\frac{\theta^\qd(\I y,r)}{\I y\theta(\I y,r)} & = -\I + \OO(\E^{-\eps y}), & 
\frac{\phi^\qd(\I y,r)}{\I y\phi(\I y,r)} & = -\I+ \OO(\E^{-\eps y}),
\end{align}
as $y\rightarrow \infty$.
\end{lemma} 

\begin{proof}
 Let $r>R$ and choose an $\varepsilon>0$ such that $\varepsilon<r-R$.
 Then the function $\Wr$ is equal to zero almost everywhere on $[r-\eps,\infty)$ and the measure $\dip$ coincides with the Lebesgue measure on $[r-\eps,\infty)$.
Therefore, we clearly have
\begin{align*}
z\theta(z,x) & = z\theta(z,r-\eps)\cos(z(x-r+\eps)) + \theta^\qd(z,r-\eps)\sin(z(x-r+\eps)),\\
\theta^\qd(z,x) & = -z\theta(z,r-\eps)\sin(z(x-r+\eps)) + \theta^\qd(z,r-\eps)\cos(z(x-r+\eps)),
\end{align*}
for all $x> r-\eps$ and $z\in\C$. 
This allows us to compute 
\begin{align*}
\frac{\theta^\qd(z,r)}{z\theta(z,r)} = \frac{-z\theta(z,r-\eps)\sin(z\eps) + \theta^\qd(z,r-\eps)\cos(z\eps)}{z\theta(z,r-\eps)\cos(z\eps) + \theta^\qd(z,r-\eps)\sin(z\eps)},
\end{align*}
which holds at least as long as $z\in\C_+$. 
Since one has  
\begin{align*}
  \frac{\sin(\I y\eps)}{\cos(\I y\eps)} = \I - 2\I\frac{\E^{-2\eps y}}{1+\E^{-2\eps y}} = \I + \OO(\E^{-2\eps y})
\end{align*}
as $y\rightarrow \infty$ and also taking into account that the function
\begin{align*}
m_\eps(z) = -\frac{\theta^\qd(z,r-\eps)}{z\theta(z,r-\eps)},\quad z\in\C\backslash\R,
\end{align*}
is a Herglotz--Nevanlinna function so that $m_\eps(\I y) = \OO(y)$ as $y\rightarrow\infty$, we get
\begin{align*}
\frac{\theta^\qd(\I y,r)}{\I y\theta(\I y,r)} &= -\I \frac{1  -\I m_\eps(\I y) + \OO(\E^{-2\eps y})}{1 - \I m_\eps(\I y) + \OO(y\E^{-2\eps y})} = -\I + \OO(\E^{-\eps y})
\end{align*}
as $y\rightarrow\infty$.
The proof of the corresponding claim for $\phi$ is just the same.
\end{proof}

We continue to denote by $\rho$ the square root of the Radon--Nikod\'ym derivative of $\dip$ and by $\dip_\sing$ the singular part of $\dip$ (both with respect to the Lebesgue measure). 

 \begin{lemma}\label{lemaasym}
The function $a$ satisfies  
\begin{align}
  \limsup_{y\rightarrow\infty} \frac{1}{y} \log|a(\I y)| = \int_0^\infty \rho(x) - 1\, dx.
\end{align} 
\end{lemma}
   
 \begin{proof}
 By utilizing the connection between generalized indefinite strings and canonical systems (see \cite[pp.~962--963]{IndefiniteString}), it follows from the Krein--de Branges theorem \cite[Theorem~4.26]{rem}, \cite[Theorem~11]{rom} that  
\begin{align*}
  \limsup_{y\rightarrow \infty}  \frac{1}{y}\log|\theta(\I y,r)|  = \int_0^r \rho(x) dx
\end{align*}
for any fixed $r>R$.
Upon denoting the function on the right-hand side of \eqref{eqnaasthephi} by $E$, we then have  
\begin{align*}
\frac{E(z)}{\theta(z,r)} & = 1 + \I \frac{\theta^\qd(z,r)}{z\theta(z,r)} - \I \frac{z\phi(z,r)}{\theta(z,r)} + \frac{\phi^\qd(z,r)}{\theta(z,r)} \\ 
&= 1 + \I \frac{\theta^\qd(z,r)}{z\theta(z,r)} + \frac{z\phi(z,r)}{\theta(z,r)}\biggl(\frac{\phi^\qd(z,r)}{z\phi(z,r)} - \I\biggr), \quad z\in\C\backslash\R.   
\end{align*}
One notices that the function
\begin{align*}
 m_r(z) = \frac{z\phi(z,r)}{\theta(z,r)}, \quad z\in\C\backslash\R, 
\end{align*}
is a Herglotz--Nevanlinna function so that $m_r(\I y)=\OO(y)$ as $y\rightarrow\infty$.
Together with Lemma \ref{lem:asymptot}, this implies the  estimate
\begin{align*}
1 \le \left|\frac{E(\I y)}{\theta(\I y,r)}\right| \le Cy
\end{align*}
for some constant $C$ and all large enough $y>0$, yielding  
\begin{align*}
 \limsup_{y\rightarrow\infty} \frac{1}{y} \log|E(\I y)| = \int_0^r \rho(x)dx. 
\end{align*} 
This readily gives 
\begin{align*}
  \limsup_{y\rightarrow\infty} \frac{1}{y} \log|a(\I y)| & = \limsup_{y\rightarrow\infty} \frac{1}{y} \log|E(\I y)| - r = \int_0^r \rho(x)-1\, dx,
\end{align*}
which proves the claim since $\rho$ is equal to one almost everywhere on $[r,\infty)$. 
 \end{proof}

With these asymptotics, we are now able to derive the main trace formula. 

\begin{lemma}\label{lemTF0}
   We have the identity 
  \begin{align}\label{eqnTF0}
\frac{2}{\pi} \int_{\R} \frac{1}{\lambda^2} \log|a(\lambda)| d\lambda  = \int_0^\infty |\Wr(x)|^2 dx  +  \int_0^\infty |\rho(x) - 1|^2 dx + \int_{[0,\infty)} d\dip_\sing. 
  \end{align}
\end{lemma}
  
  \begin{proof}
    From the representation~\eqref{eqnaasthephi} for $a$, together with the formulas in Proposition~\ref{propFS} for the fundamental system $\theta$, $\phi$, we see that $a(0)  = 1$  and after a straightforward computation furthermore that 
    \begin{align}\label{eqndddazero}
    \begin{split}
      \dot{a}(0) & = \frac{\I r}{2}-\frac{\I}{2} \biggl(\int_0^r |\Wr(x)|^2 dx +  \int_{[0,r)} d\dip\biggr)\\
      & = -\frac{\I}{2} \biggl(\int_0^\infty |\Wr(x)|^2 dx + \int_0^\infty \rho(x)^2 - 1\, dx +  \int_{[0,\infty)} d\dip_\sing\biggr)
      \end{split}
    \end{align}
    for any fixed $r>R$. 
    In particular, this yields the Taylor expansion 
      \begin{align*}
        a(z) = 1 - z \frac{\I}{2} \biggl(\int_0^\infty |\Wr(x)|^2 dx +  \int_0^\infty \rho(x)^2 - 1\, dx + \int_{[0,\infty)} d\dip_\sing\biggr) + \OO(z^2)
      \end{align*}
      as $z\rightarrow0$, which entails that       
        \begin{align}\label{eqnlogazero}
           \log|a(\lambda)| = \OO(\lambda^2)
       \end{align}
      as $\lambda\rightarrow0$ on the real line. %we get an extra order since the order $z$ term above is imaginary. 
  
   In view of the connection between generalized indefinite strings and canonical systems, it follows from \cite[Theorem~4.19]{rem} that the entire functions 
   \begin{align*}
  z & \mapsto\theta(z,x), & z&\mapsto \theta^\qd(z,x), & z & \mapsto\phi(z,x), & z & \mapsto \phi^\qd(z,x),     
\end{align*}
    belong to the Cartwright class for every $x\in[0,\infty)$ and hence so does the function $a$.
  Therefore, it admits a Nevanlinna factorization \cite[Theorem~6.13]{roro94} in the upper complex half-plane. 
  As the function $a$ has no zeros in $\C_+\cup \R$ and allows an analytic continuation across all of $\R$, this factorization takes the form 
  \begin{align*}
     a(z) = C  \exp\biggl\{ - \I \beta z + \frac{1}{\pi\I} \int_\R \biggl(\frac{1}{\lambda-z}-\frac{\lambda}{1+\lambda^2}\biggr)\log|a(\lambda)|d\lambda\biggr\}, \quad z\in\C_+,
   \end{align*}
  where $C\in\C$ is a complex constant with modulus one and $\beta$ is given by 
  \begin{align*}
    \beta = \int_0^\infty \rho(x) - 1\, dx
  \end{align*}
  in view of \cite[Theorem~6.15]{roro94} and Lemma~\ref{lemaasym}.
 Upon differentiating this expression, we get
  \begin{align*}
    \frac{\dot{a}(z)}{a(z)} = - \I \int_0^\infty \rho(x) - 1\, dx + \frac{1}{\pi\I} \int_\R \frac{1}{(\lambda-z)^2} \log|a(\lambda)|d\lambda, \quad z\in\C_+.
  \end{align*}
   Now we readily obtain identity~\eqref{eqnTF0} by letting $z$ tend to zero, employing~\eqref{eqndddazero} and noting that the limit of the integral on the right-hand side exists because of the asymptotics~\eqref{eqnlogazero}.
    \end{proof}
 
  The crucial ingredient for our proof will be an estimate on the absolutely continuous spectrum of $(L,\omega,\dip)$. 
  To this end, we first note that we have 
  \begin{align}\label{eq:m=ab}
   m(z) = \I\frac{a(z)-b(z)}{a(z) + b(z)}, \quad z\in\C_+.
  \end{align}
  Since the functions $a$ and $b$ are entire and satisfy the property~\eqref{eqnabone} on the real line, one can conclude that the spectrum of $(L,\omega,\dip)$ is purely absolutely continuous with the corresponding spectral measure $\mu$ given by
  \begin{align}
    \mu(B) = \int_B \varrho(\lambda) d\lambda
  \end{align}
  for every Borel set $B\subseteq \R$, where $\varrho$ is defined by 
   \begin{align}\label{eqnvarrho0}
    \varrho(\lambda) = \lim_{\varepsilon\rightarrow0} \frac{1}{\pi}\im\, m(\lambda+\I\varepsilon) = \frac{1}{\pi |a({\lambda})+b({\lambda})|^2}, \quad \lambda\in\R. 
   \end{align}
   We note that the function $\varrho$ is continuous and positive on $\R$.
  
    \begin{corollary}\label{corACest}
   For every compact subset $\Omega$ of $\R\backslash\{0\}$, we have the estimate 
  \begin{align}\label{eqnTFest0}
      - \frac{1}{\pi}\int_\Omega \log\big(\varrho(\lambda)C_\Omega\lambda^2\big) \frac{1}{\lambda^2} d\lambda  
                          \leq \int_0^\infty |\Wr(x)|^2 dx  + \int_0^\infty  |\rho(x) - 1|^2 dx+ \int_{[0,\infty)} d \dip_\sing,
  \end{align}
  where $C_\Omega = 4\pi (\min_{\lambda \in\Omega} |\lambda|)^{-2} $ is a positive constant.
 \end{corollary}
 
\begin{proof}
  For every real $\lambda$, we first compute that 
     \begin{align*}
       \left| 1+ \frac{a(\lambda)-b(\lambda)}{a(\lambda)+b(\lambda)}\right|^2 =  \frac{4 |a(\lambda)|^2}{|a(\lambda)+b(\lambda)|^2} = 4\pi  |a(\lambda)|^2  \varrho(\lambda) 
  \end{align*}
  and on the other side that 
       \begin{align*}
       \left| 1+ \frac{a(\lambda)-b(\lambda)}{a(\lambda)+b(\lambda)}\right|  \geq   \re\left( 1+ \frac{a(\lambda)-b(\lambda)}{a(\lambda)+b(\lambda)}\right) = 1+ \frac{1}{|a(\lambda)+b(\lambda)|^2} \geq 1.
  \end{align*}
  In combination, this gives the bound 
  \begin{align*}
    \frac{1}{|a(\lambda)|^2}  \leq 4\pi \varrho(\lambda) \leq C_\Omega \lambda^{2} \varrho(\lambda)
  \end{align*}  
  as long as $\lambda\in\Omega$, which allows us to estimate the integral  
   \begin{align*}
    - \frac{1}{\pi}\int_\Omega  \log\big(\varrho(\lambda)C_\Omega\lambda^2\big) \frac{1}{\lambda^2} d\lambda  \leq 
    \frac{1}{\pi} \int_{\Omega} \log|a(\lambda)|^2 \frac{1}{\lambda^2} d\lambda \leq 
    \frac{2}{\pi} \int_{\R} \log|a(\lambda)| \frac{1}{\lambda^2} d\lambda,
  \end{align*}
   which yields the claim in view of Lemma~\ref{lemTF0}. 
 \end{proof}

 With these auxiliary facts, we are now in position to prove our first theorem. 
 
\begin{proof}[Proof of Theorem~\ref{thm0}]
  Let us assume for now that $S$ is a generalized indefinite string $(L,\omega,\dip)$ such that $L$ is infinite and 
    \begin{align*}
      \int_0^\infty \left| \Wr(x) \right|^2 dx + \int_0^\infty |\rho(x) - 1|^2 dx + \int_{[0,\infty)} d \dip_\sing & < \infty,
    \end{align*}
    where $\Wr$ is the normalized anti-derivative of $\omega$, $\rho$ is the square root of the Radon--Nikod\'ym derivative of $\dip$ and $\dip_\sing$ is the singular part of $\dip$.
    We approximate $S$ in a suitable way by a sequence of generalized indefinite strings of the form considered before:
  For every $n\in\N$, let $(L_n,\omega_n,\dip_n)$ be the generalized indefinite string such that $L_n$ is infinite, the anti-derivative $\Wr_n$ of $\omega_n$ is equal to $\Wr$ almost everywhere on the interval $[0,n)$ as well as equal to zero almost everywhere on $[n,\infty)$ and the measure $\dip_n$ coincides with $\dip$ on the interval $[0,n)$ as well as with the Lebesgue measure on $[n,\infty)$.
   Note that by construction, we then have 
  \begin{align}\begin{split}\label{quantityConvergence2a}
	& \int_0^{\infty} |\Wr_n(x)|^2 dx+ \int_0^{\infty} |\rho_n(x) - 1|^2 dx + \int_{[0,\infty)}d \dip_{n,\sing} \\
	& \qquad\qquad \rightarrow \int_0^{\infty} |\Wr(x)|^2 dx +  \int_0^{\infty} |\rho(x) - 1|^2 dx + \int_{[0,\infty)}d \dip_\sing	
  \end{split}\end{align}  
as $n\to \infty$, where $\rho_n$ is the square root of the Radon--Nikod\'ym derivative of $\dip_n$ and $\dip_{n,\sing}$ is the singular part of $\dip_n$. 
  Furthermore,  it follows readily form \cite[Proposition~6.2]{IndefiniteString} that the corresponding Weyl--Titchmarsh functions $m_n$ converge locally uniformly to $m$. Thus the associated spectral measures $\mu_n$ certainly satisfy 
  \begin{align}\label{measureConvergence}
  	\int_\R g(\lambda)d\mu_n(\lambda) \rightarrow \int_\R g(\lambda)d\mu(\lambda), \qquad n\rightarrow\infty,
  \end{align} 
  for every continuous function $g$ on $\R$ with compact support. 

  Now take a compact set  $\Omega \subset \R\backslash \{0\}$ of positive Lebesgue measure. 
  Due to the convergence of measures $\mu_n$ in \eqref{measureConvergence} we have (see \cite[Theorem~30.2]{ba01}) 
 \begin{align*}
 \mu (\Omega) \geq \limsup_{n\to \infty} \mu_n(\Omega) = \limsup_{n\to \infty} \int_{\Omega} \varrho_n(\lambda)d\lambda,
 \end{align*}
 where the functions $\varrho_n$ are given as in~\eqref{eqnvarrho0}. 
  An application of Jensen's inequality \cite[Theorem~3.3]{ru74} then furthermore yields 
 \begin{align*}
   \mu(\Omega) \geq \limsup_{n\to \infty} D_\Omega \exp \biggl\{\frac{1}{C_\Omega D_\Omega} \int_{\Omega}\log(\varrho_n(\lambda) C_\Omega\lambda^2)\frac{1}{\lambda^2} d\lambda \biggr\},
 \end{align*}
 where $C_\Omega$, $D_\Omega$ are positive constants defined as in Corollary~\ref{corACest} and by
 \begin{align*}
   D_\Omega = \frac{1}{C_\Omega} \int_\Omega \frac{1}{\lambda^2} d\lambda.
 \end{align*} 
    In view of the estimate in Corollary \ref{corACest} and the convergence in \eqref{quantityConvergence2a}, we can conclude that $\mu(\Omega)$ is indeed positive with 
   \begin{align*}
  \mu(\Omega) & \geq   D_\Omega \exp\biggl\{\frac{-\pi}{C_\Omega D_\Omega} \biggl(\int_0^\infty |\Wr(x)|^2dx +\int_0^\infty |\rho(x)-1|^2 dx + \int_{[0,\infty)}d\dip_\sing\biggr)\biggr\}.
 \end{align*}
  Since all Borel measures on $\R$ are regular, this readily implies that $\mu(\Omega)$ is positive for every Borel set $\Omega\subseteq \R$ of positive Lebesgue measure. 
  With this fact, we have verified that the essential spectrum of $S$ coincides with the real line $\R$ and the absolutely continuous spectrum of $S$ is essentially supported on $\R$. 
  
  In order to finish the proof of Theorem~\ref{thm0}, let us suppose that $S$ is a generalized indefinite string $(L,\omega,\dip)$ such that $L$ is infinite and~\eqref{eqnCondS0} holds for a real constant $c$ and a positive constant $\eta$.
  We consider the generalized indefinite string $(L,\tilde{\omega},\tilde{\dip})$, where $\tilde{\omega}$ is defined via its normalized anti-derivative $\tilde{\Wr}$ by 
  \begin{align*}
  \tilde{\Wr} = \frac{\Wr - c}{\eta}
  \end{align*}
  and $\tilde{\dip} = \eta^{-2}\dip$. 
  Since $(L,\tilde{\omega},\tilde{\dip})$ satisfies the assumptions imposed before, we infer that the essential spectrum of $(L,\tilde{\omega},\tilde{\dip})$ coincides with the real line $\R$ and its absolutely continuous spectrum is essentially supported on $\R$. 
  However, as the corresponding Weyl-Titchmarsh functions $m$ and $\tilde{m}$ are related via  
  \begin{align*}
   m(z) = \eta\, \tilde{m}(\eta z) +c, \quad z\in\C\backslash\R, 
  \end{align*}
  the same is true for $S$. 
\end{proof}

%%%%%%%%%%%%%%%%%%%%%%%
 \section{Proof of Theorem~\ref{thmalpha}}\label{secPr2}
%%%%%%%%%%%%%%%%%%%%%%%

Before we begin, let us introduce the function $k$ on $\C_+\cup(-1,1)\cup\C_-$ by 
 \begin{align}\label{eq:defsqrt}
   k(z) = \sqrt{z^2-1}, \quad z\in\C_+\cup(-1,1)\cup\C_-,
 \end{align}
 where we continue to use the square root as it was defined in Section~\ref{secMR}.
 One notes that the function $k$ is analytic, has positive imaginary part and satisfies 
 \begin{align}\label{eqnksym}
  - k(z)^\ast = k(z^\ast), \quad z\in\C_+\cup(-1,1)\cup\C_-. 
 \end{align}
For convenience, we will extend $k$ to all of $\C$ by setting 
\begin{align}\label{eqnkonR}
  k(\lambda) = \begin{cases} -\sqrt{\lambda^2-1}, & \lambda\in(-\infty,-1], \\ \sqrt{\lambda^2-1}, & \lambda\in[1,\infty), \end{cases}
\end{align}
so that $k$ is continuous on the closed upper complex half-plane $\C_+\cup\R$. 

 We assume that $(L,\omega,\dip)$ is a generalized indefinite string such that $L$ is infinite and there is an $R>0$ such that the normalized anti-derivative $\Wr$ of $\omega$ is equal to zero almost everywhere on $[R, \infty)$ and the measure $\dip$ coincides with $\dip_1$ (recall Example~\ref{exaalpha}) on $[R,\infty)$. 
 In addition, let us also suppose that $\Wr$ and $\rho$ are equal to piecewise constant functions almost everywhere on the interval $[0,R]$ and that the support of the measure $\dip_\sing$ is a finite set, where $\rho$ is the square root of the Radon--Nikod\'ym derivative of $\dip$ and $\dip_\sing$ is the singular part of $\dip$ (both with respect to the Lebesgue measure). 
 Finally, we will also assume that $\Wr$ is equal to zero and $\rho$ is positive almost everywhere in a neighborhood of $R$. 
 The set of generalized indefinite strings defined in this way will be denoted by $\F$ for easy reference later on. 
Under these assumptions, for every $z\in\C_+\cup(-1,1)\cup\C_-$, there is a {\em Jost solution} $f(z,\redot)$ of the differential equation~\eqref{eqnDEho} such that 
\begin{align}\label{eqnJostalpha}
f(z,x) = (1+2 x)^{\frac{\I k(z)+1}{2}}, \quad x\in [R,\infty).
\end{align}
We note that since the function $f(z,\redot)$ clearly lies in $\dot{H}^1[0,\infty)$ and $L^2([0,\infty);\dip)$, the corresponding Weyl--Titchmarsh function $m$ is given by 
\begin{align}\label{eqnmasfalpha}
  m(z) = \frac{f'(z,0-)}{z f(z,0)}, \quad z\in\C\backslash\R. 
\end{align} 

Next, we define the function $a$ on  $\C_+\cup(-1,1)\cup\C_-$ by 
\begin{align}\label{eq:a_alpha}
a(z) = \frac{( \I k(z) - 1)f(z,0)+f'(z,0-)}{2\I k(z)}, \quad z\in\C_+\cup(-1,1)\cup\C_-.
\end{align}
If we denote with $\theta$, $\phi$ the fundamental system of solutions of the differential equation~\eqref{eqnDEho} as in Section~\ref{sec:prelim}, then we may write
\begin{align}
  f(z,x) = f(z,0)\theta(z,x) + f'(z,0-)\phi(z,x), \quad x\in[0,\infty),
\end{align}
for every $z\in\C_+\cup(-1,1)\cup\C_-$. 
Upon evaluating this function and its derivative at the point $R$, we get 
\begin{align}
 (1+2 R)^{\frac{\I k(z)+1}{2}} & = f(z,0)\theta(z,R) + f'(z,0-)\phi(z,R), \\
  (\I k(z)+1) (1+2 R)^{\frac{\I k(z)-1}{2}} & = f(z,0)\theta^\qd(z,R) + f'(z,0-)\phi^\qd(z,R),
\end{align}  
which we can solve for $f(z,0)$ and $f'(z,0-)$ to obtain
\begin{align}
  \label{eqnJosttp} f(z,0) & =  (1+2 R)^{\frac{\I k(z)+1}{2}} \phi^\qd(z,R) - (\I k(z)+1) (1+2 R)^{\frac{\I k(z)-1}{2}} \phi(z,R), \\
  \label{eqnJostdtp} f'(z,0-) & =  (\I k(z)+1) (1+2 R)^{\frac{\I k(z)-1}{2}}\theta(z,R) - (1+2 R)^{\frac{\I k(z)+1}{2}}\theta^\qd(z,R).
\end{align}
Plugging this into the definition of $a$ shows that  
\begin{align}\label{eq:a}
\begin{split}
	2\I k(z) (1+2 R)^{-\frac{\I k(z)+1}{2}} a(z) & = \frac{\I k(z) + 1}{1+2 R}\theta(z,R) - \theta^\qd (z,R)\\
	& \qquad +  \frac{z^2}{1+2 R}\phi(z,R) + (\I k(z) - 1)\phi^\qd(z,R),
	\end{split}
\end{align}
which guarantees that $a$ is analytic on $\C_+\cup(-1,1)\cup\C_-$. 
Moreover, we see that the function $a$ admits a unique continuation (denoted with $a$ as well for simplicity) to $\C\backslash\{-1,1\}$ such that $a$ is continuous on $\C_+\cup(-\infty,-1)\cup(-1,1)\cup(1,\infty)$. 
It can now be verified using relation~\eqref{eqnksym} that the modulus of $a$ is actually continuous on all of $\C\backslash\{-1,1\}$. 

In a similar way, we introduce the function $b$ on  $\C_+\cup(-1,1)\cup\C_-$ via 
\begin{align}\label{eqnBalpha}
	b(z) = \frac{(\I k(z) + 1)f(z,0)- f'(z,0-)}{2\I k(z)},\quad z\in\C_+\cup(-1,1)\cup\C_-.
\end{align}
As before, we see that $b$ is analytic on $\C_+\cup(-1,1)\cup\C_-$ because of the identity 
\begin{align}\label{eq:b}
\begin{split}
	2\I k(z) (1+2 R)^{-\frac{\I k(z)+1}{2}} b(z) & = -\frac{\I k(z) + 1}{1+2 R}\theta(z,R) + \theta^\qd (z,R)\\
	& \qquad -  \frac{(\I k(z)+1)^2}{1+2 R}\phi(z,R) + (\I k(z) + 1)\phi^\qd(z,R),
	\end{split}
\end{align}
and that the function $b$ admits a unique continuation to $\C\backslash\{-1,1\}$ such that $b$ is continuous on $\C_+\cup(-\infty,-1)\cup(-1,1)\cup(1,\infty)$. 
Now it is a straightforward but lengthy computation to verify that for $\lambda \in (-\infty,-1)\cup (1,\infty)$ we have  
\begin{align}\label{eq:a+b}
|a(\lambda)|^2 = |b(\lambda)|^2 + 1,
\end{align}
which guarantees that $a$ has no zeros in $(-\infty,-1)\cup(1,\infty)$. 
Furthermore, there are also no zeros of $a$ in $\C\backslash\R$. 
In fact, if $z$ was a zero in $\C\backslash \R$, then we would have  
\begin{align}
 f'(z,0-) = -(\I k(z)-1)f(z,0),
 \end{align}
 so that (recall that $m_1$ in Example~\ref{exaalpha} is a Herglotz--Nevanlinna function) 
\begin{align}
   m(z) & =  \frac{f'(z,0-)}{z f(z,0)} =  \frac{1-\I k(z)}{z} =  \frac{1}{m_1(z)},
\end{align} 
% so that we could compute the imaginary part 
%\begin{align}\begin{split}
% \im\, m(z) & = \im\, \frac{f'(z,0-)}{z f(z,0)} = \im\, \frac{1-\I k(z)}{z} \\
%   & = -\frac{\im\, z + \re\, k(z) \re\, z + \im\, k(z) \im\,z}{|z|^2} \\
%    & = - \frac{\im\,z}{|z|^2} \biggl(1 + \frac{(\re\,z)^2}{\im\, k(z)}+\im\,k(z)\biggr),
%\end{split}\end{align} % use $2\re k(z) \im k(z) = \im k(z)^2 = \im z^2 = 2\re z \im z$.
which is a contradiction to the fact that $m$ is a Herglotz--Nevanlinna function.   
This guarantees that all zeros of $a$ indeed have to lie on the interval $(-1,1)$. 
In order to prove that there are at most finitely many such zeros, we first show that the function $f(\ledot,0)$ has only finitely many zeros in $(-1,1)$. 
To this end, we note that the zeros of $f(\ledot,0)$ are precisely the solutions of the equation 
\begin{align}
 -\frac{\phi^\qd(z,R)}{z\phi(z,R)} = -\frac{1+\I k(z)}{z(1+2R)}.
\end{align}
Since the left-hand side is a meromorphic Herglotz--Nevanlinna function and the right-hand side is strictly decreasing on $(-1,1)$, we may conclude that there are only finitely many solutions of this equation  (as the function on the left-hand side must have a pole between any two solutions).  
It remains to note that the zeros of $a$ are precisely the solutions of the equation 
\begin{align}\label{eqnzerosofaalpha}
\frac{f'(z,0-)}{z f(z,0)} = \frac{1 - \I k(z)}{z}.
\end{align}
Again, since the left-hand side is a Herglotz--Nevanlinna function and the right-hand side is strictly decreasing on $(-1,0)$ and on $(0,1)$, this equation can have at most finitely many solutions (because  one can find a zero of $f(\ledot,0)$ between any two solutions of the same sign). 
In conclusion, we are able to enumerate the zeros of $a$, repeated according to multiplicity, by $\kappa_1,\ldots,\kappa_N$. 

\begin{lemma}\label{lem:asymptotalpha}
 There is an $\varepsilon>0$ such that 
\begin{align}\label{eqnasymthephi}
\frac{\theta^\qd(\I y,R)}{\I y\theta(\I y,R)} & = - \I c + \OO(\E^{-\eps y}), & 
\frac{\phi^\qd(\I y,R)}{\I y\phi(\I y,R)} & = -\I c+ \OO(\E^{-\eps y}),
\end{align}
as $y\rightarrow \infty$ for some positive constant $c$.
\end{lemma} 

\begin{proof}
 By our assumptions, there is a $\delta>0$ such that $\Wr$ is equal to zero and $\rho$ is equal to some positive constant $c$ almost everywhere on $[R-\delta,R]$.
 Furthermore, we can choose $\delta$ such that $[R-\delta,R]$ is a null set with respect to $\dip_\sing$.
 It now follows as in the proof of Lemma~\ref{lem:asymptot} that~\eqref{eqnasymthephi} holds with $\varepsilon=c\delta$. 
 \end{proof}

With this auxiliary fact, we are able to determine the asymptotics of $a$ next. 

 \begin{lemma}\label{lemaasymalpha}
The function $a$ satisfies  
\begin{align}
  \limsup_{y\rightarrow\pm\infty} \frac{1}{|y|} \log|a(\I y)| = \int_0^\infty \rho(x) - \frac{1}{1+2x}\, dx.
\end{align} 
\end{lemma}
   
 \begin{proof}
 Let us first recall that  
\begin{align*}
  \limsup_{y\rightarrow\pm\infty}  \frac{1}{|y|}\log|\theta(\I y,R)|  = \int_0^R \rho(x) dx.
\end{align*}
Upon denoting the function on the right-hand side of \eqref{eq:a} by $E$, we then have  
\begin{align*}
\frac{E(z)}{z\theta(z,R)} & = \frac{\I k(z)+1}{z(1+2R)} - \frac{\theta^\qd(z,R)}{z\theta(z,R)} + \frac{z\phi(z,R)}{(1+2R)\theta(z,R)} + \frac{(\I k(z) -1)\phi^\qd(z,R)}{z\theta(z,R)} \\ 
&= \frac{\I k(z)+1}{z(1+2R)} - \frac{\theta^\qd(z,R)}{z\theta(z,R)} + \frac{z\phi(z,R)}{\theta(z,R)}\biggl(\frac{1}{1+2R}+\frac{\I k(z) -1}{z} \frac{\phi^\qd(z,R)}{z\phi(z,R)}\biggr)  
\end{align*}
for all $z\in\C\backslash\R$. 
One notices that 
\begin{align*}
 \frac{k(\I y)}{\I |y|} = \sqrt{1+y^{-2}} = 1 + \OO(y^{-2})
\end{align*}
as $y\rightarrow\pm\infty$ and since the function
\begin{align*}
 m_R(z) = \frac{z\phi(z,R)}{\theta(z,R)}, \quad z\in\C\backslash\R, 
\end{align*}
is a Herglotz--Nevanlinna function, we also have $m_R(\I y)=\OO(|y|)$ as $y\rightarrow\pm\infty$.
Together with Lemma \ref{lem:asymptotalpha}, this implies the  estimate
\begin{align*}
\frac{1}{1+2R}\leq \Bigl| \frac{E(\I y)}{\I y\theta(\I y,R)}\Bigr| \le C |y|
\end{align*}
for some constant $C$ and all large enough $y$, yielding  
\begin{align*}
 \limsup_{y\rightarrow\pm\infty} \frac{1}{|y|} \log|E(\I y)| = \int_0^R \rho(x)dx. 
\end{align*} 
This readily gives 
\begin{align*}
  \limsup_{y\rightarrow\pm\infty} \frac{1}{|y|} \log|a(\I y)| & = \limsup_{y\rightarrow\pm\infty} \frac{1}{|y|} \log|E(\I y)| - \frac{\log(1+2R)}{2} \\
   & = \int_0^R \rho(x) - \frac{1}{1+2x}\, dx,
\end{align*}
which proves the claim since the integrand vanishes almost everywhere on $[R,\infty)$. 
 \end{proof}
 
 In order to obtain a suitable factorization of the function $a$, we first introduce the function $\zeta$ via 
\begin{align}\label{eq:map}
 \zeta(z) = z + k(z), \quad z\in\C_+\cup(-1,1)\cup\C_-.
\end{align}
We note that $\zeta$ maps the domain $\C_+\cup(-1,1)\cup\C_-$ conformally onto $\C_+$, with inverse given by the {\em Zhukovsky map}
\begin{align}
\zeta^{-1}(z) = \frac{1}{2}\biggl(z+\frac{1}{z}\biggr), \quad z\in\C_+, 
\end{align}
so that in particular one has  
\begin{align}
   k(\zeta^{-1}(z)) = z - \zeta^{-1}(z) = \frac{1}{2}\biggl(z - \frac{1}{z}\biggr), \quad z\in\C_+.
\end{align}

\begin{lemma}\label{lem:NevFactA}
The function $a$ admits the factorization
\begin{align}\begin{split}\label{eqnNevan1}
a(z) &= C \prod_{n=1}^N\frac{\zeta(z)-\zeta(\kappa_n)}{\zeta(z)- \zeta(\kappa_n)^\ast} \exp\biggl\{ - \I \beta k(z) + \frac{1}{\pi\I} \int_{|\lambda|>1} \frac{k(z)}{\lambda -z} \frac{\log|a(\lambda)|}{k(\lambda)}d\lambda\biggr\}
\end{split}\end{align}
for all $z\in\C_+\cup(-1,1)\cup\C_-$, where $C$ is a complex constant with modulus one and $\beta$ is given by
\begin{align}\label{eq:beta1}
\beta =  \int_0^\infty \rho(x) - \frac{1}{1+2x}\, dx. 
\end{align}
\end{lemma}

\begin{proof}
Define the analytic function $A$ on $\C_+$ by 
\begin{align*}
A(z) = a(\zeta^{-1}(z)),\quad z\in\C_+,
\end{align*}
and note that identity~\eqref{eq:a} then becomes
\begin{align}\label{eq:A}
\begin{split}
	& \I \biggl(z - \frac{1}{z}\biggr) (1+2 R)^{-\frac{\I}{4} (z - \frac{1}{z}) - \frac{1}{2}} A(z)  \\
	& \qquad =  \frac{\frac{\I}{2}\bigl(z - \frac{1}{z}\bigr) + 1}{1+2 R} \theta(\zeta^{-1}(z),R) - \theta^\qd (\zeta^{-1}(z),R)\\
	& \qquad\qquad  +  \frac{\zeta^{-1}(z)^2}{1+2 R} \phi(\zeta^{-1}(z),R)  +  \biggl( \frac{\I}{2}\biggl(z - \frac{1}{z}\biggr)-1\biggr)\phi^\qd(\zeta^{-1}(z),R),
	\end{split}
\end{align}
which holds for all $z\in\C_+$. We also mention that $\zeta$ maps the zeros of the function $a$ precisely to the zeros of the function $A$ so that $\zeta(\kappa_1),\ldots,\zeta(\kappa_N)$ are the zeros of $A$, repeated according to multiplicity.
Due to our assumptions on the coefficients $\omega$ and $\dip$, each of the entire functions
   \begin{align*}
  z & \mapsto\theta(z,R), & z&\mapsto \theta^\qd(z,R), & z & \mapsto z\phi(z,R), & z & \mapsto \phi^\qd(z,R),     
\end{align*}
is a linear combination of functions of the form 
\begin{align*}
  z\mapsto p(z) \E^{\I r z}
\end{align*}
with $p$ a polynomial and $r$ a real constant. 
In conjunction with this observation, we may conclude from the representation~\eqref{eq:A} that $A$ is of bounded type and hence admits a Nevanlinna factorization 
\begin{align*} 
    A(z) = C \prod_{n=1}^N\frac{z-\zeta(\kappa_n)}{z- \zeta(\kappa_n)^\ast} \exp\biggl\{ - \I \gamma_\infty z + \frac{\I\gamma_0}{z} + \frac{1}{\pi\I} \int_\R \biggl(\frac{1}{t-z}-\frac{t}{1+t^2}\biggr)\log K(t) dt\biggr\}
   \end{align*}
   for all $ z\in\C_+$, where $C$ is a complex constant with modulus one, $\gamma_\infty$ is given by
   \begin{align*}
     \gamma_\infty & = \limsup_{y\rightarrow\infty} \frac{1}{y} \log|A(\I y)| = \limsup_{y\rightarrow\infty} \frac{|\zeta^{-1}(\I y)|}{y} \frac{\log|a(\zeta^{-1}(\I y))|}{|\zeta^{-1}(\I y)|}  \\% = \frac{\beta}{2}
      & = \frac{1}{2} \int_0^\infty \rho(x) - \frac{1}{1+2x}\, dx
   \end{align*}
   in view of Lemma~\ref{lemaasymalpha}, $\gamma_0$ is similarly given by 
   \begin{align*}
       \gamma_0 & = \lim_{\varepsilon\rightarrow0} \varepsilon \log|A(\I\varepsilon)| = \lim_{\varepsilon\rightarrow0} \varepsilon|\zeta^{-1}(\I\varepsilon)| \frac{\log|a(\zeta^{-1}(\I \varepsilon))|}{|\zeta^{-1}(\I\varepsilon)|}  \\ % = \frac{\beta}{2}, 
      & = \frac{1}{2} \int_0^\infty \rho(x) - \frac{1}{1+2x}\, dx,
   \end{align*}
    and the function $K$ is given by (recall here that $|a|$ is continuous on all of $\C\backslash\{-1,1\}$)
\begin{align*}
  K(t) = \lim_{\varepsilon\rightarrow0} |A(t+\I\varepsilon)| =  \lim_{\varepsilon\rightarrow0} |a(\zeta^{-1}(t+\I\varepsilon))| = \biggl|a\biggl(\frac{1}{2}\biggl(t+\frac{1}{t}\biggr)\biggr)\biggr|
\end{align*}   
for almost all $t\in\R$.   
After a few substitutions, we find that   
\begin{align*}
 &\int_\R  \biggl(\frac{1}{t-z}-\frac{t}{1+t^2}\biggr)\log K(t)dt \\
  & \qquad = \int_{|t|<1}  \biggl(\frac{1}{t-z}-\frac{t}{1+t^2}\biggr)\log K(t)dt + \int_{|t|>1}  \biggl(\frac{1}{t-z}-\frac{t}{1+t^2}\biggr)\log K(t)dt \\
 & \qquad =  \int_{|t|>1}  \biggl(\frac{1}{1/t-z}-\frac{1/t}{1+1/t^2}\biggr)\log\,\biggl|a\biggl(\frac{1}{2}\biggl(\frac{1}{t}+t\biggr)\biggr)\biggr|  \frac{1}{t^2} dt\\
&\qquad\qquad + \int_{|t|>1}  \biggl(\frac{1}{t-z}-\frac{t}{1+t^2}\biggr)\log\,\biggl|a\biggl(\frac{1}{2}\biggl(t+\frac{1}{t}\biggr)\biggr)\biggr|dt \\
 & \qquad =  \int_{|t|>1} \frac{1-z^2}{(t-z)(1-zt)} \log\,\biggl|a\biggl(\frac{1}{2}\biggl(t+\frac{1}{t}\biggr)\biggr)\biggr| dt\\
& \qquad = \int_{|\lambda|>1} \frac{k(\zeta^{-1}(z))}{\lambda -\zeta^{-1}(z)} \frac{\log|a(\lambda)|}{k(\lambda)}d\lambda
\end{align*}
 for all $z\in\C_+$, which readily yields the claim.
   \end{proof}
   
Our factorization of the function $a$ now gives rise to the desired trace formulas.

\begin{lemma}\label{lem:TrFlaAlpha}
We have the identities 
\begin{align}\label{eqnTrAlpha1}
 -\sum_{n=1}^N \frac{\sqrt{1-\kappa_n^2}}{\kappa_n} + \frac{ 1}{\pi} \int_{|\lambda|>1} \frac{1}{\lambda^2} \frac{\log|a(\lambda)|}{k(\lambda)} d\lambda
  & = \int_{0}^{\infty}\Wr(x) dx,\\
  \frac{1}{2}\sum_{n=1}^N \log\frac{1+\sqrt{1-\kappa_n^2}}{1-\sqrt{1-\kappa_n^2}}  - \frac{1}{\pi} \int_{|\lambda|>1} \frac{1}{\lambda} \frac{\log|a(\lambda)|}{k(\lambda)} d\lambda & =  \int_0^\infty \rho(x) - \frac{1}{1+2x}\, dx,\label{eqnTrAlpha1B}
\end{align}
and
\begin{align}\label{eqnTrAlpha2}
\begin{split}
& \sum_{n=1}^N \frac{\sqrt{1-\kappa_n^2}}{\kappa_n^2}  + \frac{1}{2}\sum_{n=1}^N \log\frac{1-\sqrt{1-\kappa_n^2}}{1+\sqrt{1-\kappa_n^2}} 
+ \frac{2}{\pi } \int_{|\lambda|>1}  \frac{\lambda^2-1}{\lambda^3}  \frac{\log|a(\lambda)|}{k(\lambda)} d\lambda \\
& \qquad =  \int_0^\infty |\Wr(x)|^2 (1+2x) dx+\int_0^\infty \Bigl|\rho(x) - \frac{1}{1+2x}\Bigr|^2 (1+2 x) dx  \\
&\qquad\qquad + \int_{[0,\infty)} (1+2 x) d \dip_\sing(x).
\end{split}
\end{align}
\end{lemma}   
  
\begin{proof} 
Evaluating~\eqref{eq:a} at zero and noting that $k(0) = \I$, we first see that  $a(0)=1$.
After differentiating~\eqref{eq:a}, evaluating at zero, noting that $\dot{k}(0) = 0$ and employing the formulas in Proposition~\ref{propFS}, we get that 
\begin{align}\label{eq:dota0}
\dot{a}(0) = \int_{0}^{\infty}\Wr(x) dx.
\end{align}
Differentiating~\eqref{eq:a} once more, evaluating at zero, noting that $\ddot{k}(0) = -\I$ and using the formulas in Proposition~\ref{propFS} again finally gives   
\begin{align}\label{eq:ddota0}
\begin{split}
\ddot{a}(0) & = \biggl(\int_0^\infty \Wr(x)dx\biggr)^2 -  \int_0^\infty |\Wr(x)|^2 (1+2 x) dx \\
& \qquad - \int_{0}^{R} \rho(x)^2(1+2 x) dx + \frac{\log(1+2R)}{2} - \int_{[0,\infty)} (1+2x) d\dip_\sing(x).
\end{split}
\end{align}

On the other side, we evaluate~\eqref{eqnNevan1} at zero and take absolute values to obtain
\begin{align*}
\beta = \frac{1}{2}\sum_{n=1}^N \log\frac{1+\sqrt{1-\kappa_n^2}}{1-\sqrt{1-\kappa_n^2}}  - \frac{1}{\pi} \int_{|\lambda|>1} \frac{1}{\lambda} \frac{\log|a(\lambda)|}{k(\lambda)} d\lambda,
\end{align*}
which yields identity~\eqref{eqnTrAlpha1B} in view of~\eqref{eq:beta1}.
Using~\eqref{eqnNevan1} to compute the logarithmic derivative of $a$, we arrive at
\begin{align}\begin{split}\label{eqnLogDeralpha}
  \frac{\dot{a}(z)}{a(z)} & = \sum_{n=1}^N \frac{\dot{\zeta}(z)}{\zeta(z)} \frac{\I\sqrt{1-\kappa_n^2}}{z-\kappa_n} -  \I\beta \dot{k}(z) + \frac{1}{\pi\I} \int_{|\lambda|>1} \frac{\dot{k}(z)}{\lambda-z} \frac{\log|a(\lambda)|}{k(\lambda)}d\lambda \\
   & \qquad + \frac{1}{\pi\I} \int_{|\lambda|>1} \frac{k(z)}{(\lambda-z)^2} \frac{\log|a(\lambda)|}{k(\lambda)} d\lambda
\end{split}\end{align}
for all $z$ close enough to zero (so that $a(z)$ is non-zero). 
Evaluating at zero gives
\begin{align*}
  \dot{a}(0) = -\sum_{n=1}^N \frac{\sqrt{1-\kappa_n^2}}{\kappa_n} + \frac{ 1}{\pi} \int_{|\lambda|>1} \frac{1}{\lambda^2} \frac{\log|a(\lambda)|}{k(\lambda)} d\lambda
\end{align*}
and thus the  first trace formula \eqref{eqnTrAlpha1} upon taking into account~\eqref{eq:dota0}.
By differentiating~\eqref{eqnLogDeralpha} and then evaluating at zero again, we get 
\begin{align*}
\ddot{a}(0) - \dot{a}(0)^2  & = -\sum_{n=1}^N \frac{\sqrt{1-\kappa_n^2}}{\kappa_n^2} -\beta +  \frac{1}{\pi} \int_{|\lambda|>1} \frac{2-\lambda^2}{\lambda^3} \frac{\log|a(\lambda)|}{k(\lambda)} d\lambda.
\end{align*}
After adding $2\beta$ on both sides and using formula~\eqref{eqnTrAlpha1B} to replace $\beta$ on the right-hand side, we end up with 
\begin{align*}
\ddot{a}(0) - \dot{a}(0)^2  + 2\beta & = -\sum_{n=1}^N \frac{\sqrt{1-\kappa_n^2}}{\kappa_n^2} +\frac{1}{2} \sum_{n=1}^N  \log\frac{1+\sqrt{1-\kappa_n^2}}{1-\sqrt{1-\kappa_n^2}} \\
 & \qquad +  \frac{2}{\pi} \int_{|\lambda|>1} \frac{1-\lambda^2}{\lambda^3} \frac{\log|a(\lambda)|}{k(\lambda)} d\lambda.
\end{align*}
In order to obtain the last trace formula~\eqref{eqnTrAlpha2}, it now remains to note that the left-hand side is equal to minus the right-hand side of~\eqref{eqnTrAlpha2} upon taking into account~\eqref{eq:ddota0}, \eqref{eq:dota0} and~\eqref{eq:beta1}. 
\end{proof}

We note that the integral term on the left-hand side of the identity~\eqref{eqnTrAlpha2} is non-negative in view of~\eqref{eq:a+b} and~\eqref{eqnkonR}. 
 In particular, this observation will allow us to obtain a Lieb--Thirring-type estimate on the eigenvalues of the corresponding self-adjoint realization in the interval $(-1,1)$. 
 To this end, we first point out that the spectrum in $(-1,1)$ is purely discrete since~\eqref{eqnmasfalpha} shows that the Weyl--Titchmarsh function $m$ allows a meromorphic extension to $\C_+\cup(-1,1)\cup\C_-$.
 Moreover, there are only finitely many eigenvalues in $(-1,1)$ as they are precisely the zeros of the function $f(\ledot,0)$. 
 Because zero is certainly not an eigenvalue (since $L$ is infinite), we may enumerate all eigenvalues in $(-1,1)$ in the following way:
\begin{align}
-1<\lambda_{K_-}^- <\dots < \lambda_1^-<0<\lambda_1^+ < \dots<\lambda^+_{K_+}<1.
\end{align}

  \begin{corollary}\label{corEVestAlpha}
   We have the estimate 
  \begin{align}
\begin{split}\label{eqnLTalpha}
 & \frac{2}{3}\sum_{i=1}^{K_-} \big(1-|\lambda_i^-|^{2}\big)^{3/2} + \frac{2}{3}\sum_{i=1}^{K_+} \big(1-|\lambda_i^+|^2\big)^{3/2} \\  
 & \qquad \leq  \int_0^\infty |\Wr(x)|^2 (1+2x)dx   +\int_0^\infty \Bigl|\rho(x) - \frac{1}{1+2x}\Bigr|^2 (1+2 x) dx \\
 & \qquad\qquad + \int_{[0,\infty)} (1+2 x) d \dip_\sing(x).
\end{split}  \end{align}
 \end{corollary}

\begin{proof}
Since the left-hand side of~\eqref{eqnzerosofaalpha} is a Herglotz--Nevanlinna function and the right-hand side is strictly decreasing on $(-1,0)$ and on $(0,1)$ with a pole at zero, we may find a solution of equation~\eqref{eqnzerosofaalpha} between any two eigenvalues (poles of the left-hand side) as well as between any eigenvalue and zero. 
As the solutions of this equation are precisely  the zeros $\kappa_1,\ldots,\kappa_N$ of the function $a$, we may conclude that 
\begin{align*}
 \lambda_{K_-}^- < \kappa_{n_-(K_-)} < \lambda_{K_--1}^- < \cdots < \lambda_{1}^- < \kappa_{n_-(1)} < 0
\end{align*}
as well as
\begin{align*}
0 < \kappa_{n_+(1)} <  \lambda_{1}^+ < \cdots < \lambda_{K_+-1}^+ < \kappa_{n_+(K_+)} < \lambda_{K_+}^+
\end{align*}
for some indices $n_-(K_-),\ldots,n_-(1),n_+(1),\ldots,n_+(K_+)\in\{1,\ldots,N\}$. 

Now let us consider the function $G$ defined by 
  \begin{align}\label{eq:funtrace}
    G(s) =   \frac{s}{1-s^2} + \frac{1}{2} \log \left(\frac{1 - s}{1 + s}\right), \quad s\in[0,1),
 \end{align}
 and first notice that $G$ is strictly increasing on $[0,1)$ and positive on $(0,1)$ since we may compute 
\begin{align*}
 G'(s) = \frac{1+s^2}{(1-s^2)^2} -\frac{1}{1-s^2} = \frac{2s^2}{(1-s^2)^2}, \quad s\in(0,1). 
\end{align*}
Moreover, the function $G$ satisfies the bound 
\begin{align*}
G(s) = \int_0^s G'(r) dr = \int_0^s \frac{2r^2}{(1-r^2)^2}dr \geq \int_0^s 2r^2dr = \frac{2}{3}s^3, \quad s\in(0,1). 
\end{align*}
By combining all these facts, we can estimate
\begin{align*}
  \frac{2}{3} \big(1-|\lambda_i^\pm|^{2}\big)^{3/2} \leq G\Bigl(\sqrt{1 - |\lambda_i^\pm|^2}\Bigr) < G\Bigl(\sqrt{1 - \kappa_{n_\pm(i)}^2}\Bigr)
\end{align*}
for all $i\in\{1,\ldots,K_\pm\}$. 
This allows us to bound the left-hand side of~\eqref{eqnLTalpha} by 
\begin{align*}
  \sum_{i=1}^{K_-} G\Bigl(\sqrt{1 - \kappa_{n_-(i)}^2}\Bigr) +  \sum_{i=1}^{K_+} G\Bigl(\sqrt{1 - \kappa_{n_+(i)}^2}\Bigr) \leq  \sum_{n=1}^{N} G\bigl(\sqrt{1 - \kappa_{n}^2}\bigr).
\end{align*}
It is readily seen that this sum coincides with the first two terms in~\eqref{eqnTrAlpha2}, which yields the claim as the integral term on the left-hand side there is non-negative. 
\end{proof}

We are now going to use the identity~\eqref{eqnTrAlpha2} to estimate the absolutely continuous spectrum of  $(L,\omega,\dip)$. 
 To this end, we first note that we have  
\begin{align}\label{eq:m=ab2}
m(z) = \frac{1}{z} + \frac{\I k(z)}{z}\frac{a(z) - b(z)}{a(z)+b(z)}, \quad z\in\C\backslash\R. 
\end{align} 
 Since the functions $k$, $a$ and $b$ are continuous on $\C_+\cup(-\infty,-1)\cup(1,\infty)$ and satisfy~\eqref{eq:a+b} for real $\lambda$ with $|\lambda|> 1$, % that is, the Weyl--Titchmarsh function $m$ has a continuous extension to $\C_+\cup(-\infty,-1)\cup(1,\infty)$
 one can conclude that the spectrum of $(L,\omega,\dip)$ on the set $(-\infty,-1)\cup(1,\infty)$ is purely absolutely continuous with the corresponding spectral measure $\mu$ given by 
  \begin{align}
    \mu(B) = \int_B \varrho(\lambda) d\lambda
  \end{align}
  for every Borel set $B\subseteq(-\infty,-1)\cup(1,\infty)$, where $\varrho$ is defined by 
  \begin{align}\label{eqnrhoabalpha}
    \varrho(\lambda) = \lim_{\varepsilon\rightarrow0} \frac{1}{\pi}\im\, m(\lambda+\I\varepsilon) 
    =  \frac{k(\lambda)}{\pi\lambda|a(\lambda)+b(\lambda)|^2}, \quad \lambda\in(-\infty,-1)\cup(1,\infty). 
   \end{align} 
   We note that the function $\varrho$ is continuous and positive on $(-\infty,-1)\cup(1,\infty)$.
   
 \begin{corollary}\label{corACest2}
For every compact subset $\Omega$ of $(-\infty,-1)\cup(1,\infty)$, we have the estimate 
  \begin{align}\label{eqnTFest02}
  \begin{split}
     &  - \frac{1}{\pi} \int_\Omega \log\biggl(\varrho(\lambda)\frac{4\pi\lambda^3}{k(\lambda)}\biggl) \frac{k(\lambda)}{\lambda^3} d\lambda \\
     & \qquad  \leq  \int_0^\infty |\Wr(x)|^2 (1+2x) dx +\int_0^\infty \Bigl|\rho(x) - \frac{1}{1+2x}\Bigr|^2 (1+2 x)dx \\
       & \qquad\qquad  + \int_{[0,\infty)} (1+2 x) d \dip_\sing(x).
       \end{split}
  \end{align}
 \end{corollary} 
 
 \begin{proof}
   For every $\lambda\in (-\infty,-1)\cup(1,\infty)$, we first compute that 
     \begin{align*}
       \left| 1+ \frac{a(\lambda)-b(\lambda)}{a(\lambda)+b(\lambda)}\right|^2 =  \frac{4 |a(\lambda)|^2}{|a(\lambda)+b(\lambda)|^2} 
       =  4\pi \frac{\lambda}{k(\lambda)}\varrho(\lambda) |a(\lambda)|^2
  \end{align*}
  and on the other side that 
       \begin{align*}
       \left| 1+ \frac{a(\lambda)-b(\lambda)}{a(\lambda)+b(\lambda)}\right|  \geq   
       \re\left( 1+ \frac{a(\lambda)-b(\lambda)}{a(\lambda)+b(\lambda)}\right) = 1+\frac{1}{|a(\lambda)+b(\lambda)|^2} \geq 1.
  \end{align*}
  In combination, this gives the bound 
  \begin{align*}
    \frac{1}{|a(\lambda)|^2}  \leq  \frac{4\pi\lambda}{k(\lambda)}\varrho(\lambda) 
       \leq   \frac{4\pi\lambda^3}{k(\lambda)}\varrho(\lambda),  \end{align*}  
  which allows us to estimate the integral  
  \begin{align*}
    - \frac{1}{\pi} \int_\Omega \log\biggl(\varrho(\lambda)\frac{4\pi\lambda^3}{k(\lambda)}\biggl) \frac{k(\lambda)}{\lambda^3} d\lambda & \leq \frac{1}{\pi} \int_{\Omega} \log|a(\lambda)|^2 \frac{k(\lambda)}{\lambda^3}  d\lambda\\
     & \leq \frac{2}{\pi} \int_{|\lambda|>1} \frac{\lambda^2-1}{\lambda^3} \frac{\log|a(\lambda)|}{k(\lambda)} d\lambda.
   \end{align*}
 Now it remains to notice that the sum of the first two terms in~\eqref{eqnTrAlpha2} is non-negative since the function $G$ defined in~\eqref{eq:funtrace} takes positive values.   
 \end{proof}

 We are now ready to prove our second main result.
     
\begin{proof}[Proof of Theorem~\ref{thmalpha}]
   Let us assume for now that $S$ is a generalized indefinite string $(L,\omega,\dip)$ such that $L$ is infinite and 
    \begin{align*}
      \int_0^\infty | \Wr(x)|^2 x\, dx   +\int_0^\infty   \Bigl| \rho(x) - \frac{1}{1+2x} \Bigr|^2 x\, dx + \int_{[0,\infty)}  x\, d\dip_\sing(x) < \infty,
    \end{align*}
    where $\Wr$ is the normalized anti-derivative of $\omega$, $\rho$ is the square root of the Radon--Nikod\'ym derivative of $\dip$ and $\dip_\sing$ is the singular part of $\dip$. 
  We are first going to construct a suitable approximating sequence of generalized indefinite strings $(L_n,\omega_n,\dip_n)$ from the set $\mathcal{F}$. 
  For every $n\in\N$, let $L_n$ be infinite and choose $R_n>n$ such that 
  \begin{align*}
    \int_{R_n}^\infty |\Wr(x)|^2 x\, dx +  \int_{R_n}^\infty   \Bigl| \rho(x) - \frac{1}{1+2x} \Bigr|^2 x\, dx< \frac{1}{n}.
  \end{align*}
   We can then find a real-valued, piecewise constant function $\Wr_n$ on $[0,\infty)$ such that  
  \begin{align*}
   \int_0^{R_n} |\Wr_n(x)-\Wr(x)|^2 dx & < \frac{1}{n R_n}
  \end{align*}
  and such that $\Wr_n$ is equal to zero to the right of $R_n-\varepsilon_n$ for some $\varepsilon_n>0$. 
  The distribution $\omega_n$ is now defined in such a way that the corresponding normalized anti-derivative coincides with $\Wr_n$ almost everywhere. 
  We can also find a non-negative function $\rho_n$ on $[0,\infty)$ which is piecewise constant on the interval $[0,R_n]$ with 
  \begin{align*}
   \int_0^{R_n} |\rho_n(x)-\rho(x)|^2 dx & < \frac{1}{n R_n}
  \end{align*}
  and which is given explicitly by 
  \begin{align*}
    \rho_n(x) & = \frac{1}{1+2x}
   \end{align*}
   for all $x>R_n$.    
   Moreover, it is possible to choose $\rho_n$ such that it is positive in a neighborhood of $R_n$. 
  Apart from this, we are able to find a non-negative Borel measure $\dip_{n,\sing}$ which is supported on a finite set contained in $[0,R_n)$ with 
  \begin{align*}
   \int_{[0,\infty)} d\dip_{n,\sing} & = \int_{[0,\infty)} d\dip_\sing, & \int_{[0,\infty)} x\, d\dip_{n,\sing}(x) & \leq \int_{[0,\infty)} x\, d\dip_\sing(x),
  \end{align*}
  and such that for almost every $x\in[0,\infty)$ we have 
  \begin{align*}
    \int_{[0,x)} d\dip_{n,\sing} \rightarrow \int_{[0,x)} d\dip_\sing, \qquad n\rightarrow \infty.
  \end{align*} 
    % Divide the interval $[0,n)$ into $n^2$ subintervals and put the mass of $\dip_\sing$ on each of these intervals into the left endpoint. 
  % Then put the mass of $\dip_\sing$ on the interval $[n,\infty)$ into the point $n$. 
  The measure $\dip_n$ is then defined by 
  \begin{align*}
    \int_B d\dip_n = \int_B \rho_n(x)^2 dx + \int_B d\dip_{n,\sing}
  \end{align*}
  for every Borel set $B\subseteq[0,\infty)$. 
  Note that by construction, the generalized indefinite strings $(L_n,\omega_n,\dip_n)$ belong to the set $\mathcal{F}$ and the quantities 
  \begin{align}\label{eqnNconvalpha}
  \begin{split}
     & \int_0^\infty |\Wr_n(x)|^2 (1+2x) dx + \int_0^\infty   \Bigl| \rho_n(x) - \frac{1}{1+2x} \Bigr|^2 (1+2x) dx  \\
     & \qquad\qquad\qquad\qquad\qquad\qquad\qquad\qquad\qquad   + \int_{[0,\infty)} (1+2x) d\dip_{n,\sing}(x)
     \end{split}
    \end{align}
    are bounded by a positive constant $M$ for all $n\in\N$. 
 Furthermore, it follows readily from \cite[Proposition~6.2]{IndefiniteString} that the corresponding Weyl--Titchmarsh functions $m_n$ converge locally uniformly to $m$. 
  Thus the associated spectral measures $\mu_n$ certainly satisfy 
  \begin{align}\label{eqnmuconvalpha}
   \int_\R g(\lambda)d\mu_n(\lambda) \rightarrow \int_\R g(\lambda) d\mu(\lambda), \qquad n\rightarrow\infty,
  \end{align}
  for every continuous function $g$ on $\R$ with compact support. 

   In order to prove that the essential spectrum of $S$ is restricted to $(-\infty,-1]\cup[1,\infty)$, let $I$ be a compact interval in $(-1,1)$. 
   Because of the estimate in Corollary~\ref{corEVestAlpha} and boundedness of the quantities in~\eqref{eqnNconvalpha}, we see that there is an integer $K_I$ such that $(L_n,\omega_n,\dip_n)$ has at most $K_I$ eigenvalues in the interval $I$ for every $n\in\N$. 
   It now follows from the convergence of the measures $\mu_n$ in~\eqref{eqnmuconvalpha} that the limit measure $\mu$ is supported on a finite set on $I$, which implies that $S$ has at most finitely many eigenvalues in $I$.  
  Since the interval $I$ was arbitrary, we conclude that the essential spectrum of $S$ is necessarily contained in  $(-\infty,-1]\cup[1,\infty)$. 
  
  Now take a compact set $\Omega\subset(-\infty,-1)\cup(1,\infty)$ of positive Lebesgue measure.
   Due to  the convergence of the measures $\mu_n$ in~\eqref{eqnmuconvalpha} we have (see \cite[Theorem~30.2]{ba01})
  \begin{align*}
   \mu(\Omega) \geq \limsup_{n\rightarrow\infty} \mu_n(\Omega) = \limsup_{n\rightarrow\infty} \int_\Omega \varrho_n(\lambda)d\lambda,
  \end{align*}
  where the functions $\varrho_n$ are given as in~\eqref{eqnrhoabalpha}. 
 An application of Jensen's inequality \cite[Theorem~3.3]{ru74} then furthermore yields
 \begin{align*}
  \mu(\Omega) & \geq \limsup_{n\rightarrow\infty} D_\Omega \exp\biggl\{\frac{1}{4\pi D_\Omega} \int_\Omega \log\biggl(\varrho_n(\lambda)\frac{4\pi\lambda^3}{k(\lambda)}\biggr) \frac{k(\lambda)}{\lambda^3}  d\lambda\biggr\}, 
 \end{align*}
 where $D_\Omega$ is a positive constant defined by
 \begin{align*}
  D_\Omega= \frac{1}{4\pi}\int_\Omega   \frac{k(\lambda)}{\lambda^3}  d\lambda.
 \end{align*}
  In view of the estimate in Corollary~\ref{corACest2} and the bound on~\eqref{eqnNconvalpha}, we conclude that 
   \begin{align*}
  \mu(\Omega) & \geq  D_\Omega \E^{\frac{- M}{4 D_\Omega}} > 0.
 \end{align*}
  Since all Borel measures on $\R$ are regular, this readily implies that $\mu(\Omega)$ is positive for every Borel set $\Omega\subseteq(-\infty,-1]\cup[1,\infty)$ of positive Lebesgue measure. 
  Thus, we have finally verified that the essential spectrum of $S$ coincides with the set $(-\infty,-1]\cup[1,\infty)$ and the absolutely continuous spectrum of $S$ is essentially supported on $(-\infty,-1]\cup[1,\infty)$. 
  
 In order to finish the proof of Theorem~\ref{thmalpha}, let us suppose  that $S$ is a generalized indefinite string $(L,\omega,\dip)$ such that  $L$ is infinite and~\eqref{eqnCondSalpha} holds for a real constant $c$ and positive constants $\alpha$ and $\eta$. 
 We consider the generalized indefinite string $(L,\tilde{\omega},\tilde{\dip})$, where $\tilde{\omega}$ is defined via its normalized anti-derivative $\tilde{\Wr}$ by
  \begin{align*}
    \tilde{\Wr}(x) & = \frac{\Wr(x/\alpha) - c}{\eta}
  \end{align*}
  for almost all $x\in[0,\infty)$ and $\tilde{\dip}$ is defined by 
  \begin{align*}
    \int_B d\tilde{\dip} = \frac{\alpha}{\eta^2} \int_{B/\alpha} d\dip = \frac{1}{\eta^2} \int_B \rho(x/\alpha)^2 dx + \frac{\alpha}{\eta^2} \int_{B/\alpha} d\dip_\sing
  \end{align*} 
  for every Borel set $B\subseteq[0,\infty)$. 
  Since $(L,\tilde{\omega},\tilde{\dip})$ satisfies the assumptions imposed before, we infer that the essential spectrum of $(L,\tilde{\omega},\tilde{\dip})$ coincides with the set $(-\infty,-1]\cup[1,\infty)$ and its absolutely continuous spectrum is essentially supported on $(-\infty,-1]\cup[1,\infty)$. 
  However, as the corresponding Weyl--Titchmarsh functions $m$ and $\tilde{m}$ are related via 
  \begin{align*}
    m(z) = \eta\, \tilde{m}(\eta z/\alpha) + c, \quad z\in\C\backslash\R, 
  \end{align*} 
 we see that the essential spectrum of $S$ coincides with the set $(-\infty,-\alpha/\eta]\cup[\alpha/\eta,\infty)$ and its absolutely continuous spectrum is essentially supported on $(-\infty,-\alpha/\eta]\cup[\alpha/\eta,\infty)$.
\end{proof}

%%%%%%%%%%%%%%%%%%%%%
 \section{The conservative Camassa--Holm flow}\label{secAPP}
 %%%%%%%%%%%%%%%%%%%%%

  In this section, we are going to demonstrate how our results apply to the isospectral problem of the conservative Camassa--Holm flow. 
  To this end, let $u$ be a real-valued function in $H^1_\loc[0,\infty)$ and $\dip$ be a non-negative Borel measure on $[0,\infty)$. 
  We define the distribution $\omega$ in $H^{-1}_\loc[0,\infty)$ by  
\begin{align}\label{eqnDefomega}
 \omega(h) = \int_0^\infty u(x)h(x)dx + \int_0^\infty u'(x)h'(x)dx, \quad h\in H^1_\cc[0,\infty),
\end{align}
so that $\omega = u - u''$ in a distributional sense. 
 Now the isospectral problem of the conservative Camassa--Holm flow is associated with the differential equation
 \begin{align}\label{eqnCHISP}
 -g'' + \frac{1}{4} g = z\, \omega\, g + z^2 \dip\, g,
\end{align}
where $z$ is a spectral parameter.
 Just like for generalized indefinite strings, this differential equation has to be understood in a weak sense in general:   
  A solution of~\eqref{eqnCHISP} is a function $g\in H^1_{\loc}[0,\infty)$ such that 
 \begin{align}\label{eqnDEweakform}
   \Delta_g h(0) + \int_{0}^\infty g'(x) h'(x) dx + \frac{1}{4} \int_0^\infty g(x)h(x)dx = z\, \omega(gh) + z^2 \dip(g h) 
 \end{align} 
 for some constant $\Delta_g\in\C$ and every function $h\in H^1_\cc[0,\infty)$.
 For such a solution $g$, the constant $\Delta_g$ is uniquely determined and will be denoted with $g'(0-)$. 
 
It has been demonstrated in \cite[Section~7]{ACSpec} that it is always possible to transform the differential equation~\eqref{eqnCHISP} into the differential equation
    \begin{align}\label{eqntildeString}
    - f'' = z\,\tilde{\omega} f + z^2\tilde{\dip} f
   \end{align}
   for some corresponding generalized indefinite string $(\infty,\tilde{\omega},\tilde{\dip})$.
 To this end, let us introduce the diffeomorphism $\Sr\colon[0,\infty)\rightarrow[0,\infty)$ by  
 \begin{align}\label{eq:Sdiff}
  \Sr(t) = \log(1+t), \quad t\in[0,\infty),
 \end{align} 
 and define a real-valued measurable function $\tilde{\Wr}$ on $[0,\infty)$ such that   
 \begin{align}\label{eqnDefa}
   \tilde{\Wr}(t)  =  u(0) - \frac{u'(\Sr(t))+ u(\Sr(t))}{1+t}  
 \end{align}
 for almost all $t\in[0,\infty)$, where we note that the right-hand side is well-defined almost everywhere.
 It follows readily that the function $\tilde{\Wr}$ is locally square integrable, so that we can find a real distribution $\tilde{\omega}$ in $H^{-1}_\loc[0,\infty)$ which has $\tilde{\Wr}$ as its normalized anti-derivative. 
 Furthermore, the non-negative Borel measure $\tilde{\dip}$ on $[0,\infty)$ is given by setting 
 \begin{align}\label{eqnDefbeta}
   \tilde{\dip}(B) =  \int_B \frac{1}{1+t}\, d\dip\circ \Sr(t)  = \int_{\Sr(B)} \E^{-x} d\dip(x)
 \end{align}
 for every Borel set $B\subseteq[0,\infty)$.
 This defines a generalized indefinite string $(\infty,\tilde{\omega},\tilde{\dip})$ whose relation to the differential equation~\eqref{eqnCHISP} can be described as follows: 

 \begin{lemma}\label{lem:sol=sol}
   A function $g$ is a solution of the differential equation~\eqref{eqnCHISP} if and only if the function $f$ defined by 
   \begin{align}\label{eqnfg}
     f(t) =   g(\Sr(t)) \sqrt{1+t}, \quad t\in[0,\infty),
   \end{align}
   is a solution of the differential equation~\eqref{eqntildeString}. 
 \end{lemma}

We now define the Weyl--Titchmarsh function $m$ associated with~\eqref{eqnCHISP}  by 
 \begin{align}
  m(z) =  \frac{\psi'\NLz}{z\psi(z,0)},\quad z\in\C\backslash\R,
 \end{align} 
  where $\psi(z,\redot)$ is the (up to scalar multiples) unique non-trivial solution of the differential equation~\eqref{eqnCHISP} which lies in $H^1[0,\infty)$ and $L^2([0,\infty);\dip)$, guaranteed to exist by~\cite[Corollary 7.2]{ACSpec}. 
 In view of Lemma \ref{lem:sol=sol}, we readily compute that 
 \begin{align}\label{eqnmms}
   m(z) = \tilde{m}(z) - \frac{1}{2z}, \quad z\in\C\backslash\R, 
 \end{align}
where $\tilde{m}$ is the Weyl--Titchmarsh function of the corresponding generalized indefinite string $(\infty,\tilde{\omega},\tilde{\dip})$.  
 In particular, we see that $m$ is a Herglotz--Nevanlinna function and the Borel measure $\mu$ in the corresponding integral representation differs from the one for the generalized indefinite string only by a point mass at zero. 
 The measure $\mu$ is a spectral measure for a suitable self-adjoint realization $\T$ of the spectral problem~\eqref{eqnCHISP}; compare \cite{bebrwe08, LeftDefiniteSL, CHPencil}. 
 Since this establishes an immediate connection between the spectral properties of $\T$ and the corresponding generalized indefinite string $(\infty,\tilde{\omega},\tilde{\dip})$, we may now invoke Theorem~\ref{thmalpha}. 
 
 \begin{theorem}\label{thmApp2CH}
  If the function $u$ belongs to $H^1[0,\infty)$, the singular part $\dip_\sing$ of the measure $\dip$ is finite and $\rho-1$ belongs to $L^2[0,\infty)$, where $\rho$ is the square root of the Radon--Nikod\'ym derivative of $\dip$, then the essential spectrum of $\T$ coincides with the set $(-\infty,-1/2]\cup[1/2,\infty)$ and the absolutely continuous spectrum of $\T$ is essentially supported on $(-\infty,-1/2]\cup[1/2,\infty)$. 
 \end{theorem}
 
 \begin{proof} 
  Under these assumptions, we readily see that the coefficients of the corresponding generalized indefinite string $(\infty,\tilde{\omega},\tilde{\dip})$ satisfy 
 \begin{align*}
  \int_0^\infty |\tilde{\Wr}(t)  -u(0) |^2 t\, dt & \leq \int_0^\infty |u'(\Sr(t))+ u(\Sr(t))|^2 \frac{1}{1+t} dt  \\
     & = \int_0^\infty |u'(x)+ u(x)|^2 dx < \infty,
  \end{align*}
  upon performing a substitution, as well as 
   \begin{align*}
 \int_0^\infty \Bigl|\tilde{\rho}(t)  - \frac{1}{1+t} \Bigr|^2t\, dt  &=  \int_0^\infty \Bigl|\frac{\rho(\Sr(t))}{1+t}  - \frac{1}{1+t} \Bigr|^2t\, dt \\
  & \le  \int_0^\infty |\rho(\Sr(t))   - 1 |^2 \frac{1}{1+t}dt \\
  & =  \int_0^\infty | \rho(x)  - 1|^2 dx < \infty,
    \end{align*}
    where $\tilde{\rho}$ is the square root of the Radon--Nikod\'ym derivative of $\tilde{\dip}$, and 
  \begin{align*}
   \int_{[0,\infty)} t\, d\tilde{\dip}_\sing(t) = \int_{[0,\infty)} \frac{t}{1+t} d\dip_\sing\circ\Sr(t) \leq \int_{[0,\infty)} d\dip_\sing\circ\Sr = \int_{[0,\infty)} d\dip_\sing <\infty,
 \end{align*}
 where $\tilde{\dip}_\sing$ is the singular part of $\tilde{\dip}$.
  Now the claim follows from Theorem~\ref{thmalpha} with $c=u(0)$, $\alpha=1/2$ and $\eta=1$. 
 \end{proof}

Of course, it is also desirable to consider the spectral problem for~\eqref{eqnCHISP} on the whole real line. 
In this case, one can show,  using a standard argument based on stability of the absolutely continuous spectrum under finite rank perturbations, that the essential spectrum coincides with the set $(-\infty,-1/2]\cup[1/2,\infty)$ and the absolutely continuous spectrum is of multiplicity two and essentially supported on $(-\infty,-1/2]\cup[1/2,\infty)$ if $u$ is a real-valued function in $H^1(\R)$ and $\dip$ is a non-negative Borel measure on $\R$ such that its singular part is finite and $\rho-1$ belongs to $L^2(\R)$, where $\rho$ is the square root of the Radon--Nikod\'ym derivative of $\dip$;  see Theorem~\ref{thm2CHac}.

\bigskip
\noindent
{\bf Acknowledgments.}
We thank Rossen Ivanov for useful discussions and hints with respect to the literature.

\end{document}